\documentclass[11pt,letterpaper]{article}
\usepackage{graphicx}
\usepackage{epsfig,subfigure}
\usepackage{amssymb}
\usepackage{amsthm}
\usepackage{thmtools}
\usepackage{amsfonts}
\usepackage{amsmath}
\usepackage{multicol,multirow}
\usepackage{algorithm,algcompatible}
\usepackage{url}
\usepackage[margin=1in]{geometry}
\usepackage{lipsum}
\usepackage{changepage}

\newtheorem{theorem}{Theorem}
\newtheorem{lemma}[theorem]{Lemma}
\newtheorem{corollary}[theorem]{Corollary}
\newtheorem{proposition}[theorem]{Proposition}

\newtheorem{conjecture}[theorem]{Conjecture}

\theoremstyle{definition}
\newtheorem{example}[theorem]{Example}

\usepackage{xcolor}
\usepackage{ifthen}
\usepackage{etoolbox}

\definecolor{antiquebronze}{rgb}{0.4,0.36,0.12}
\definecolor{revisionBlue}{rgb}{0,0,0}
\definecolor{red}{rgb}{1,0,0}

\usepackage{bm} 
\providecommand{\mathbold}[1]{\bm{#1}}
\usepackage{xfrac}

\newcommand{\vct}[1]{\bm{#1}}
\newcommand{\mtx}[1]{\mathbold{#1}}
\newcommand{\R}{\mathbb{R}}
\newcommand{\N}{\mathbb{N}}
\newcommand{\Q}{\mathbb{Q}}

\newcommand{\evec}{\vct{a}}
\newcommand{\emat}{\mtx{A}}
\newcommand{\eentry}{a}

\newcommand{\newpoly}[1]{\mathcal{P}(#1)}
\newcommand{\cnns}[1]{\mathsf{C_{NNS}}(#1)}
\newcommand{\cnnp}[1]{\mathsf{C_{NNP}}(#1)}
\newcommand{\csage}[1]{\mathsf{C_{SAGE}}(#1)}
\newcommand{\cage}[1]{\mathsf{C_{AGE}}(#1)}
\newcommand{\cpolyage}[1]{\mathsf{C_{AGE}^{POLY}}(#1)}
\newcommand{\cpolysage}[1]{\mathsf{C_{SAGE}^{POLY}}(#1)}
\newcommand{\relent}[2]{D\left(#1, #2\right)}

\newcommand{\sagerelax}[2]{{#1}_{\mathsf{SAGE}}^{(#2)}}

\newcommand{\constrsagedual}[2]{{(#1, #2)}_{\mathsf{SAGE}}^{}}

\DeclareMathOperator{\supp}{supp}

\DeclareMathOperator{\Sig}{\mathsf{Sig}}
\DeclareMathOperator{\Poly}{\mathsf{Pol}}

\DeclareMathOperator{\conv}{conv}
\DeclareMathOperator{\cone}{co}
\DeclareMathOperator{\ext}{ext}
\DeclareMathOperator{\sint}{int}
\DeclareMathOperator{\cl}{cl}

\usepackage{hyperref}  
\hypersetup{
    colorlinks=true,
    linkcolor=blue,
    filecolor=magenta,      
    urlcolor=cyan,
    citecolor=blue,
}

\title{Newton Polytopes and Relative Entropy Optimization}

\author{Riley Murray$^\dag$, Venkat Chandrasekaran$^{\dag,\ddag}$, and Adam Wierman$^\dag$ \thanks{Email: rmurray@caltech.edu, venkatc@caltech.edu, adamw@caltech.edu} \vspace{0.25in} \\ $^\dag$ Department of Computing and Mathematical Sciences\\ $^\ddag$ Department of Electrical Engineering \\ California Institute of Technology \\ Pasadena, CA 91125}

\date{October 2, 2018; revised May 12, 2020}

\begin{document}

\maketitle



\begin{abstract}
	
	\vspace{1em}
    
    {\color{revisionBlue}Certifying function nonnegativity is a ubiquitous problem in computational mathematics, with especially notable applications in optimization.  We study the question of certifying nonnegativity of signomials based on the recently proposed approach of Sums-of-AM/GM-Exponentials (SAGE) decomposition due to the second author and Shah.  The existence of a SAGE decomposition is a sufficient condition for nonnegativity of a signomial, and it can be verified by solving a tractable convex relative entropy program.  We present new structural properties of SAGE certificates such as a characterization of the extreme rays of the cones associated to these decompositions as well as an appealing form of sparsity preservation.  These lead to a number of important consequences such as conditions under which signomial nonnegativity is equivalent to the existence of a SAGE decomposition; our results represent the broadest-known class of nonconvex signomial optimization problems that can be solved efficiently via convex relaxation.  The analysis in this paper proceeds by leveraging the interaction between the convex duality underlying SAGE certificates and the face structure of Newton polytopes.  While our primary focus is on signomials, we also discuss how our results provide efficient methods for certifying polynomial nonnegativity, with complexity independent of the degree of a polynomial.}
	
	\vspace{1em}
	\noindent \textbf{Keywords:} arithmetic-geometric-mean inequality, certifying nonnegativity, fewnomials, SAGE, signomials, sparse polynomials
\end{abstract}

\clearpage

\section{Introduction}

{\color{revisionBlue}The problem of certifying function nonnegativity is broadly applicable
in the mathematical sciences.  Optimization provides an especially notable example: since $f^\star = \inf_{\vct{x} \in \R^n} f(\vct{x}) $ can be expressed as $f^\star = \sup\{ \gamma \,:\, f(\vct{x}) - \gamma \geq 0 ~ \forall \vct{x} \in \R^n\}$, any certificate that $f - \gamma$ is nonnegative on $\R^n$ is a proof that $\gamma$ is a lower bound on $f^\star$.  Other tasks that can framed as nonnegativity certification problems include deciding stability of a dynamical system, bounding Lipschitz constants, or determining the consistency of nonlinear equations.
It is computationally intractable to decide function nonnegativity in general, but we may yet seek \textit{sufficient conditions} for nonnegativity which are tractable to certify.
This article is concerned with the Sums-of-AM/GM-Exponentials or ``SAGE'' certificates of signomial nonnegativity \cite{CS16}.

Signomials are functions of the form $\vct{x} \mapsto \sum_{i=1}^m c_i \exp(\evec_i^\intercal \vct{x})$, where $\vct{c} \in \R^m$ is the coefficient vector and the $\evec_i$ are called \textit{exponent vectors}.
Signomial nonnegativity is NP-Hard to decide by reduction to matrix copositivity \cite{Murty1987}.
Global signomial optimization is tractable in the special case of convex \textit{geometric programming} \cite{Boyd2007}, however high-fidelity signomial models often do not fit into this framework \cite{Dembo1978-sig-ChemE, KajMikael2002-sig-ChemE,jabr2007-sig-EE,Kwan2007-sig-EE,Kirschen2018b-sig-aero,Hall2018-sig-aero}.
Our study of SAGE certificates is motivated by desires to better understand signomials as a fundamental class of functions, and to provide greater ability to cope with the nonconvex signomial programs which fall outside the scope of geometric programming.

There are two perspectives with which one can look at SAGE certificates.
On the one hand, there is an interpretation where the summands are certified as nonnegative by a weighted AM/GM inequality.
Equivalently, the summands may be certified as nonnegative by the principle of strong duality as applied to a suitable convex program.
The connection to convex duality is essential, as it provides a way to tractably represent the SAGE cone via the \textit{relative entropy function}
\[
\relent{\vct{\nu}}{\vct{\lambda}} = \textstyle\sum_{i=1}^m \nu_i \ln(\nu_i / \lambda_i), ~~~ \vct{\nu},\vct{\lambda} \in \mathbb{R}^m_+.
\]
Over the course of this article we employ convex analysis and convex geometry in an effort to better understand fundamental aspects of SAGE certificates.
The convex-analytic arguments are used to establish structural results concerning the SAGE cone itself-- for example, the nature of its extreme rays, and sparsity-preservation properties of SAGE decompositions.
The convex-geometric considerations come into play when we ask precisely how the SAGE cone relates to the nonnegativity cone.
Here we use \textit{Newton polytopes}, which are convex hulls of the exponents $\{\evec_i\}_{i=1}^m$, often jointly with consideration to the sign pattern of a signomial's coefficient vector.

Analysis by Newton polytopes has a longer history in study of polynomials, and in particular in the study of sparse polynomials.
Prominent examples in this area include Khovanskii's fewnomials \cite{fewnomialsFirst,fewnomialsBook}, Reznick's agiforms \cite{extremalPSDFormsWithFewTerms,Reznick1989}, and Bajar and Stein's work on polynomial coercivity \cite{NewtonCoercivePolys}.
Many such works ``signomialize'' polynomials via a substitution $x_j \leftarrow \exp y_j$ in certain intermediate proofs.
We adopt a different perspective, where signomials are the first-class object.
This perspective allows us to properly study signomials as a class of functions distinct from polynomials, while also providing a transparent mechanism to obtain polynomial results down the line.
}  

\subsection{Our contributions}
{\color{revisionBlue}
Section \ref{sec:SAGE} briefly reviews relevant background on SAGE certificates for signomial nonnegativity.

In Section \ref{sec:SAGEcontributions} we prove a number of new structural properties of SAGE certificates.  Theorem \ref{thm:restrictSS} proves an important sparsity-preserving property of SAGE functions.
If a signomial $f$ is SAGE, then there exists a decomposition $f = \sum_k f_k$ that certifies this property (i.e., each $f_k$ is an ``AGE function''; certifiably nonnegative by a relative entropy inequality, as described in Section \ref{sec:SAGE}), such that each $f_k$ consists only of those terms that appear in $f$.
Furthermore, the process of summing the $f_k$ to obtain $f$ results in no cancellation of coefficients on basis functions $\vct{x} \mapsto \exp(\evec_i^\intercal \vct{x})$.
Theorem \ref{thm:extreme_rays} goes on to provide a characterization of the extreme rays of the cone of SAGE functions; in particular, all nontrivial extreme rays are given by AGE functions that are supported on simplicial Newton polytopes.

Section \ref{sec:structureSAGEandNNG} leverages the understanding from Section \ref{sec:SAGEcontributions} to derive a collection of structural results which describe
when nonnegative signomials are SAGE, with the Newton polytope being the primary subject of these theorems' hypotheses.
Theorem \ref{thm:linear_indep_and_all_nonext_neg_is_exact} is concerned with cases where the Newton polytope is simplicial, while Theorems \ref{thm:needNonnegOfFaces} and \ref{thm:SageEqNng} concern when it ``decomposes'' in an appropriate sense.
Each of these theorems exhibits invariance under nonsingular affine transformations of the exponent vectors.
Corollaries \ref{cor:simplicial_optimization_zero_anywhere} and \ref{cor:constrained_case_exact} show how Theorem \ref{thm:linear_indep_and_all_nonext_neg_is_exact} applies to signomial optimization problems.
We conclude the section with a result on conditions under which SAGE can recognize signomials that are bounded below (Theorem \ref{thm:fsageisbounded1pluseps}).

In Section \ref{sec:PolyNNG} we specialize our results on signomials to polynomials, by defining a suitable ``signomial repesentative'' of a polynomial, and requiring that the signomial admit a SAGE decomposition.
The resulting class of ``SAGE polynomials'' inherits a tractable representation from the cone of SAGE signomials (Theorem \ref{thm:poly_sage_tractable}) as well as other structural properties on sparsity preservation and extreme rays (Corollaries \ref{cor:poly_sparse_decomp} and \ref{cor:extreme_rays_polynomial}).
Moving from a polynomial to a signomial representative is simple but somewhat delicate, yielding both stronger (Corollary \ref{corr:lin_indep_sage_poly}) and weaker (Corollary \ref{cor:nearlySagePolyEqNNP}) results than in the signomial case.
We then situate our results on polynomial SAGE certificates in the broader literature, with specific emphasis on ``Sums-of-Nonnegative-Circuits'' (SONC, see Sections \ref{subsec:amgm} and \ref{subsec:sonc}) and ``Sums-of-Squares'' (SOS, see Section \ref{subsec:sos}).
The section is concluded with discussion on how our results provide the basis for a sparsity-preserving hierarchy of convex relaxations for polynomial optimization problems.

Section \ref{sec:constructCounterExAndDualChar} demonstrates that there are meaningful senses in which our results from Section \ref{sec:structureSAGEandNNG} cannot be improved upon.
Through Theorem \ref{thm:S_alpha_subset_of_THING_is_exact}, we provide a novel dual characterization of conditions under which the SAGE cone and the cone of nonnegative signomials coincide.
}  

\subsection{Related work}

\subsubsection{Signomials and signomial optimization}
{\color{revisionBlue}
The literature on signomials is quite fragmented, owing to a wide range of conventions used for this class of functions across fields and over time.
A recurring theme is to consider signomials as functions $\vct{y} \mapsto \sum_{i=1}^m c_i \vct{y}^{\evec_i}$, where monomials $\vct{y}^{\evec_i} \doteq \prod_{j=1}^n y_j^{\eentry_{ij}}$ remain well-defined for real $\eentry_{ij}$ so long as $\vct{y}$ is restricted to the positive orthant.
In analysis of biochemical reaction networks, signomials are often called \textit{generalized polynomials} \cite{Mller2015,Mller2018}, or simply ``polynomials over the positive orthant'' \cite{Pantea2012}.
In amoeba theory one usually calls signomials \textit{exponential sums} \cite{Silipo2012,FdW2017}.
Much of the earlier optimization literature referred to signomial programming as ``generalized geometric programming,'' but this term now means something quite different \cite{Boyd2007}.

In the taxonomy of optimization problems, geometric programming is to signomial optimization what convex quadratic programming is to polynomial optimization.
Current approaches to global signomial optimization use successive linear or geometric programming approximations together with branch-and-bound \cite{sigopt1,sigopt2,sigopt3}.
Equality constrained signomial programs are often treated by penalty or augmented Lagrangian methods \cite{Rountree1982}, and are notoriously difficult to solve \cite{Yan1976,OCW2017}.
The SAGE approach to signomial optimization does not involve branch-and-bound, and does not entail added complexity when considering signomial equations instead of inequalities.}

\subsubsection{Sums-of-Squares certificates and polynomial optimization} \label{subsec:related-sos}

\textcolor{revisionBlue}{There is a large body of work on SOS certificates for polynomial nonnegativity, and the resulting convex relaxations for polynomial optimization problems \cite{parilloPhD,Lasserre2001,shor}.
Over the course of this article we make two contributions which have direct parallels in the SOS literature.}

\textcolor{revisionBlue}{Our results in Section \ref{sec:SAGEcontributions} are along the lines of David Hilbert's 1888 classification of the number of variables ``$n$'' and the degrees ``$2d$'' for which SOS-representability coincides with polynomial nonnegativity \cite{Hilbert1888}.
The granularity with which we seek such a classification is distinct from that in the SOS literature, as there is no canonical method to take finite-dimensional subspaces of the infinite-dimensional space of signomials.}

\textcolor{revisionBlue}{A principal drawback of the SOS method is that its canonical formulation requires a semidefinite matrix variable of order ${n + d \choose d}$ -- and the size of this matrix is exponential in the degree $d$.
In Section \ref{sec:PolyNNG} we use SAGE signomials to certify polynomial nonnegativity in a way which is unaffected by the polynomial's degree.
Section \ref{subsec:sos} compares our proposed method to SOS, as well as refinements and variations of SOS which have appeared in the literature: \cite{sparsesosDecideSupports2004,sparsesosChordal2006,sparsesosDemmel2008,SDSOS1}.
}  

\subsubsection{Certifying function nonnegativity via the AM/GM inequality} \label{subsec:related-amgm}

\textcolor{revisionBlue}{The ``AGE functions'' in a SAGE decomposition may be proven nonnegative in either of two ways.
The first approach is to certify a particular relative entropy inequality over a signomial's coefficients, which is described in Section \ref{sec:SAGE}.
This approach is known to be computationally tractable, 
and it provides a convenient tool for proving structural results for the set of SAGE certificates.
The second method is to find weights for an appropriate AM/GM inequality over a signomial's coefficients; this latter method directly connects SAGE to a larger literature on certifying function nonnegativity via the AM/GM inequality, which we summarize in the sequel.}

\textcolor{revisionBlue}{The earliest systematic theoretical studies in this area were undertaken by Reznick \cite{extremalPSDFormsWithFewTerms,Reznick1989} in the late 1970s and 1980s.
The first developments of any computational flavor came from P{\'{e}}bay, Rojas and Thompson in 2009 \cite{Pbay2009}, via their study of polynomial maximization.
P{\'{e}}bay et al. used tropical geometry and $\mathcal{A}$-discriminants to develop an AM/GM type certificate for boundedness of functions supported on matroid-theoretic circuits.
In this context, a function is supported on a circuit if the monomial exponents $\{\evec_i\,:\, c_i \neq 0 \}$ form a minimal affinely-dependent set.
In 2011 Ghasemi et al. pioneered the use of geometric programming to recognize functions which were certifiably nonnegative by the AM/GM inequality and a sums-of-binomial-squares representation \cite{GandM,GandMandL}.
In 2012, Paneta, Koeppl, and Craciun derived an AM/GM condition to certify $\R^n_+$-nonnegativity of polynomials supported on circuits \cite[Theorem 3.6]{Pantea2012}.
Follow-up work by August, Craciun, and Koeppl used \cite[Theorem 3.6]{Pantea2012} to determine invariant sets for biological dynamical systems \cite{August2012}.
A short while later, Iliman and de Wolff suggested taking sums of globally nonnegative circuit polynomials \cite{SONC1};
the resulting \textit{SONC polynomials} have since become an established topic in the literature \cite{SONC2,SONC3,SONCExp,wangAGEpolynomial,wangSoncSupports}.}

\textcolor{revisionBlue}{We continue to make connections to the AM/GM-certificate literature throughout this article;  \cite{Pantea2012,SONC1,SONC3,SONCExp} are revisited in Section \ref{subsec:amgm}, and \cite{SONC1,wangAGEpolynomial,wangSoncSupports} are addressed in Section \ref{subsec:sonc}.}

\subsection{Notation and conventions}

{\color{revisionBlue}Special sets include $[\ell] \doteq \{1,\ldots,\ell\} \subset \N$, the probability simplex $\Delta_\ell \subset \mathbb{R}^\ell$, the nonnegative reals $\R_+$, and the positive reals $\R_{++}$.}
Vectors and matrices always appear in boldface.
{\color{revisionBlue}For indexing reasons we find it useful to define certain vectors by writing $\vct{v} = (v_i)_{i\in[\ell]\setminus k}$; such a vector is said to belong to ``$\R^{[\ell]\setminus k}$.''
If we wish to drop a the $k^{\text{th}}$ component from a vector $\vct{v} \in \R^\ell$, we write $\vct{v}_{\setminus k} \doteq (v_i)_{i \in [\ell]\setminus k}$.
The \textit{support} of a vector $\vct{v} \in \R^\ell$ is $\supp \vct{v} \doteq \{ i \in [\ell] \,:\, v_i \neq 0 \}$.}
We use $\vct{e}_i$ to denote the $i^{\text{th}}$ standard basis vector in $\mathbb{R}^\ell$ {\color{revisionBlue}with $\ell$ inferred from context, and set $\vct{1} = \sum_{i=1}^\ell \vct{e}_i$}.
{\color{revisionBlue}The operator $\oplus$ is used for vector concatenation.
We often call finite a point-set $\{ \vct{u}_i \}_{i=1}^\ell$ \textit{simplicial} if it is affinely independent.}

The operators ``$\cl$'' and ``$\conv$'' 
{\color{revisionBlue} return}
a set's closure and convex hull respectively; ``$\ext$'' returns the extreme points of a compact convex set.  
We use $U + V$ to denote the Minkowski sum of sets $U$ and $V$ within $\mathbb{R}^\ell$. 
For any convex cone $K$ contained in $\mathbb{R}^\ell$ there is an associated \textit{dual cone} $K^\dagger \doteq \{ \vct{y} : \vct{y}^\intercal \vct{x} \geq 0 \text{ for all } \vct{x} \text{ in } K \}.$

We reserve $\emat$ for a real $n \times m$ matrix with \textcolor{revisionBlue}{distinct} columns $\{\evec_i\}_{i=1}^m$, and $\vct{c}$ for a vector in $\mathbb{R}^m$.
Writing $f = \Sig(\emat, \vct{c})$ means that $f$ takes values $f(\vct{x}) = \sum_{i=1}^m c_i \exp( \evec_i^\intercal \vct{x})$. 
For a fixed signomial $\Sig(\emat,\vct{c})$, we often refer to $\evec_i \in \mathbb{R}^n$ as an \textit{exponent vector}, and use $\newpoly{\emat}$ to denote its Newton polytope.
The cone of {\color{revisionBlue} coefficients for} nonnegative signomials over exponents $\emat$ is
\[
\cnns{\emat} = \{\, \vct{c} \,:\, \Sig(\emat,\vct{c})(\vct{x}) \geq 0 \text{ for all } \vct{x} \text{ in } \mathbb{R}^n \}.
\]
{\color{revisionBlue} We sometimes overload terminology and refer to $\cnns{\emat}$ as a cone of signomials, rather than a cone of coefficients.
We routinely work with matrices ``$\emat_{\setminus k}$'' formed by deleting some $k^{\text{th}}$ column from $\emat$.
The matrix $\emat_{\setminus k}$ is applied to a vector $\vct{\nu} \in \R^{[m] \setminus k}$ by $\emat_{\setminus k} \vct{\nu} = \sum_{i \in [m]\setminus \{k\}} \evec_i \nu_i$.}

\section{Background Theory on SAGE Functions}\label{sec:SAGE}

In their debut, SAGE functions were used as a building block for a hierarchy of convex relaxations to challenging nonconvex signomial optimization problems \cite{CS16}. 
Underlying this entire hierarchy were the simple facts that SAGE functions are globally nonnegative, and efficiently recognizable. 
The purpose of this section is to review the theory of SAGE functions to the extent that it is needed for subsequent development. 

Section \ref{subsec:age} introduces the idea of an \textit{AGE function}, which serve as the building blocks of SAGE functions. The way in which AGE functions extend to SAGE functions is given in Section \ref{subsec:sum_of_age}.
Section \ref{subsec:nng_to_opt} describes the connection between nonnegativity and optimization in the context of SAGE relaxations. These three sections {\color{revisionBlue}are essential for} understanding \textcolor{revisionBlue}{the present} article.

\subsection{
AM/GM Exponentials
}\label{subsec:age}
We need some additional structure to make it easier to verify membership in the nonnegativity cone $\cnns{\emat}$. 
The structure used by Chandrasekaran and Shah \cite{CS16} was that the coefficient vector $\vct{c}$ contained at most one negative entry $c_k$; if such a function was globally nonnegative, they called it an \textit{AM/GM Exponential}, or an \textit{AGE function}.
To facilitate \textcolor{revisionBlue}{the} study of such functions, \cite{CS16} defines the \textit{$k^{\text{th}}$ AGE cone}
\begin{equation}
\cage{\emat,k} = \{  \vct{c} : \vct{c}_{\setminus k} \geq \vct{0} \text{ and } \vct{c} \text{ belongs to } \cnns{\emat}    \}. \label{eq:defAGE}
\end{equation}
\textcolor{revisionBlue}{Elements of $\cage{\emat,k}$ are sometimes called \textit{AGE vectors}.} 

{\color{revisionBlue}
It is evident that $\cage{\emat,i}$ is a proper convex cone which contains the nonnegative orthant.
By using a convex duality argument, one may show that} a vector $\vct{c}$ with $\vct{c}_{\setminus k} \geq \vct{0}$ belongs to $\cage{\emat,k}$ if and only if
{\color{revisionBlue}
\begin{equation}
\text{some} \quad \vct{\nu} \in \mathbb{R}^{[m]\setminus k}_{+} \quad \text{satisfies} \quad [\emat_{\setminus k} - \evec_k \vct{1}^\intercal]\vct{\nu} = \vct{0} \quad \text {and}  \quad  \relent{\vct{\nu}}{e \vct{c}_{\setminus k}}  \leq c_k.
\label{eq:ageRelEnt}
\end{equation}}
It is crucial that the representation in \eqref{eq:ageRelEnt} is jointly convex in $\vct{c}$ and the auxiliary variable $\vct{\nu}$, and moreover that no assumption is made on the sign of $c_k$.
{\color{revisionBlue}Using the representation \eqref{eq:ageRelEnt}, one may derive the following expression for} the dual of the {\color{revisionBlue}$k^{\text{th}}$} AGE cone
\begin{align}
\cage{\emat,k}^\dagger = \cl\{ \vct{v} :
&~{\color{revisionBlue}\vct{v} > 0}, \text{ and for some } \vct{\mu}_k \text{ in } \mathbb{R}^n \text{ we have} \nonumber \\
&v_k \ln(v_k/v_i) \leq (\evec_k- \evec_i)^\intercal \vct{\mu}_k \text{ for } i \text{ in } [m]  \}. \label{eq:cAGEstar}
\end{align}
\textcolor{revisionBlue}{The ``size'' of a primal or dual AGE cone refers to the number of variables plus the number of constraints in the above representations, which is $O(m)$ assuming $n \leq m$.}

\subsection{From ``AGE'' to ``SAGE''}\label{subsec:sum_of_age}

\textcolor{revisionBlue}{We call a signomial \textit{SAGE} if it can be written as a sum of AGE functions.
SAGE functions are globally nonnegative by construction;
cones of coefficients for SAGE signomials are denoted by
\begin{equation}
\csage{\emat} \doteq \textstyle\sum_{k=1}^m \cage{\emat,k}.\label{eq:CSAGE}
\end{equation}
Standard calculations in conic duality yield the following expression for a dual SAGE cone
\begin{equation}
\csage{\emat}^\dagger = \cap_{k=1}^m \cage{\emat,k}^\dagger. \label{eq:CSAGEstar}
\end{equation}}
\textcolor{revisionBlue}{Equations \ref{eq:CSAGE} and \ref{eq:CSAGEstar} provide natural definitions, but they also contain redundancies.
\begin{proposition}{\cite[Section 2.4]{CS16}}\label{prop:sage_def_wrt_exts}
    If $\evec_k \in \ext\newpoly{\emat}$ and $f = \Sig(\emat,\vct{c})$ is nonnegative, then $c_k \geq 0$.
    Consequently, if $\evec_k \in \ext\newpoly{\emat}$ then $\cage{\emat,k} = \R^m_+$.
\end{proposition}
Proposition \ref{prop:sage_def_wrt_exts} is the most basic way Newton polytopes appear in the analysis of nonnegative signomials.
In our context it means that so long as $\ext \newpoly{\emat} \subsetneq \{\evec_i\}_{i=1}^m$, we can take $\csage{\emat}$ as the Minkowski sum of AGE cones $\cage{\emat,k}$ where $\evec_k$ are nonextremal in $\newpoly{\emat}$.}

\subsection{From nonnegativity to optimization}\label{subsec:nng_to_opt}

A means for certifying nonnegativity \textcolor{revisionBlue}{provides} a natural method for computing lower bounds for minimization problems: 
given a function $f$, find the largest constant $\gamma$ where $f - \gamma$ can be certified as globally nonnegative. 
We \textcolor{revisionBlue}{now} formalize this procedure for signomials and SAGE certificates. 

Given a signomial $f = \Sig(\emat,\vct{c})$ and a constant $\gamma$ in $\mathbb{R}$, we want to check if $ f - \gamma$ is SAGE. 
To do this, we need an unambiguous representation of $f - \gamma$ in terms of $\Sig(\cdot,\cdot)$ notation. 
Towards this end we \textcolor{revisionBlue}{suppose} $\emat$ has $\evec_1 = \vct{0}$, and we make a point of allowing any entry of $\vct{c}$ to be zero.
Under these  assumptions, the function $f - \gamma$ can be written as $\Sig(\emat, \vct{c} - \gamma \vct{e}_1)$. 
Thus the optimization problem
\begin{equation}
f_{\mathsf{SAGE}} \doteq \sup\{~ \gamma ~:~ \vct{c} - \gamma \vct{e}_1 \text{ in } \csage{\emat}  \} \label{eq:defSage0Primal}
\end{equation}
is well defined, and its optimal value satisfies $f_{\mathsf{SAGE}} \leq f^\star$. We also analyze the dual problem to \eqref{eq:defSage0Primal}, and for reference, we obtain via conic duality that
\begin{equation}
f_{\mathsf{SAGE}} = \inf\{~ \vct{c}^\intercal \vct{v} : ~\vct{e}_1^\intercal \vct{v} = 1,~ \vct{v} \text{ in } \csage{\emat}^\dagger   \}. \label{eq:defSage0Dual_conic}
\end{equation}

Although not done in \cite{CS16}, it can be shown that strong duality holds in the primal-dual pair \eqref{eq:defSage0Primal}-\eqref{eq:defSage0Dual_conic}.
This fact is important for our later theorems, and so we make a point to state it clearly in the following proposition:
\begin{proposition}
	Strong duality always holds in the computation of $f_{\mathsf{SAGE}}$.\label{prop:strong_duality}
\end{proposition}
The proposition is proven in Appendix 7.2 using convex analysis.

\section{Structural Results for SAGE Certificates}\label{sec:SAGEcontributions}

This section presents {\color{revisionBlue}two} new geometric results and analytical characterizations on the SAGE cone.
These results have applications to polynomial nonnegativity, as discussed later in Section \ref{sec:PolyNNG}.
Statements of the theorems are provided below along with {\color{revisionBlue}remarks on} the theorems' significance. Proofs are deferred to later subsections.

\subsection{Summary of structural results}

Our first theorem shows that when checking if $\vct{c}$ belongs to $\csage{\emat}$ we can restrict the search space of SAGE decompositions to those exhibiting a very particular structure. It highlights the sparsity-preserving property of SAGE, and in so doing has significant implications for both the practicality of solving SAGE relaxations, and Section \ref{sec:PolyNNG}'s development of SAGE polynomials.
\begin{theorem}\label{thm:restrictSS}
    If $\vct{c}$ is a vector in $\csage{\emat}$ with nonempty \textcolor{revisionBlue}{index} set $N \doteq \{ i : c_i < 0\}$, then {\color{revisionBlue}there} exist vectors {\color{revisionBlue}$\{ \vct{c}^{(i)} \in \cage{\emat,i}\}_{i \in N}$} satisfying $\vct{c} = \sum_{i \in N} \vct{c}^{(i)}$ and $c^{(i)}_j = 0$ for all distinct $i,j \in N$.
\end{theorem}
We can use Theorem \ref{thm:restrictSS} to define some parameterized AGE cones that will be of use to us in Section \ref{sec:structureSAGEandNNG}. 
Specifically,  for an index set \textcolor{revisionBlue}{$J$} contained within $[m]$, and an index $i$ in $[m]$, define
\begin{equation}
    \cage{\emat,i,\textcolor{revisionBlue}{J}} = \{  \vct{c} : \vct{c} \text{ in } \cage{\emat,i}, ~ c_j = 0 \text{ for all } j \text{ in } \textcolor{revisionBlue}{J} \setminus \{i\} \}. \nonumber
\end{equation}
In terms of such sets we have the following corollary of Theorem \ref{thm:restrictSS}.
\begin{corollary}
	A signomial $f = \Sig(\emat,\vct{c})$ with $\evec_1 = \vct{0}$  has 
	\[f_{\mathsf{SAGE}} = \sup\{ \gamma : \vct{c} - \gamma \vct{e}_1 \mathrm{~in~} \textstyle\sum_{i \in N\cup \{1\} }\cage{\emat,i,N} \}\]
	for both $N = \{ i: c_i < 0 \}$ and $N = \{ i : c_i \leq 0 \}$.
	\label{cor:sagerelax_restrictss_primal}
\end{corollary}
This corollary has two implications {\color{revisionBlue}concerning} practical algorithms for signomial optimization.
First, it shows that for $k = | \{ i : c_i < 0 \}|$, computing $f_{\mathsf{SAGE}}$ can easily be accomplished with a relative entropy program of size $O(km)$; this is a dramatic improvement over the na\"ive implementation for computing $f_{\mathsf{SAGE}}$, which involves a relative entropy program of size $O(m^2)$.
Second -- the improved conditioning resulting from restricting the search space in this way often makes the difference in whether existing solvers can handle SAGE relaxations of moderate size.
This point is {\color{revisionBlue}highlighted} in recent experimental demonstrations of relative entropy relaxations; the authors of \cite{REPOP} discuss various preprocessing strategies to {\color{revisionBlue}more quickly solve} such optimization problems.

Our next theorem characterizes the extreme rays of the SAGE cone. To describe these extreme rays, we use a notion from matroid theory \cite{matroids}:
a set of points $X = \{ \vct{x}_i \}_{i=1}^\ell$ is called a \textit{circuit} if it is affinely dependent, but any proper subset $\{ \vct{x}_i \}_{i \neq k}$ is affinely independent. If {\color{revisionBlue}the convex hull of} a circuit with $\ell$ elements contains $\ell-1$ extreme points, then we say the circuit is \textit{simplicial}.
\begin{theorem}\label{thm:extreme_rays}
	{\color{revisionBlue}If $\vct{c} \in \R^m$ generates an extreme ray of $\csage{\emat}$, then $\{ \evec_i \,:\, i \in [m], c_i \neq 0 \}$ is either a singleton or a simplicial circuit.}
\end{theorem}

{\color{revisionBlue}Theorem \ref{thm:extreme_rays} can be viewed as a signomial generalization of a result by Reznick concerning \textit{agiforms} \cite[Theorem 7.1]{Reznick1989}. The theorem} admits a partial converse:
{\color{revisionBlue}if $X = \{\evec_j\}_{j \in J\cup\{i\}}$ is a simplicial circuit with nonextremal term $\evec_i$, then there is an extreme ray of $\cage{\emat,i}$ supported on $J \cup \{i\}$.}
When specialized to the context of polynomials, this result gives us an equivalence between SAGE polynomials (suitably defined in Section \ref{sec:PolyNNG}) and the previously defined SONC polynomials \cite{SONC1}, thus providing an efficient description of the latter set which was not known to be tractable.

\subsection{Proof of the restriction theorem for SAGE decompositions (Theorem \ref{thm:restrictSS})}

\textcolor{revisionBlue}{Our proof requires two lemmas. The first such lemma indicates the claim of the theorem applies far more broadly than for SAGE functions alone.}

\begin{lemma}\label{lem:k_neg_coeff_means_k_age_functions}
	{\color{revisionBlue}Let $K \subset \R^m$ be a convex cone containing the nonnegative orthant. For an index $i \in [m]$, define $C_i = \{ \vct{c} \in K\,:\, \vct{c}_{\setminus i} \geq \vct{0}\}$, and sum these to $C = \sum_{i=1}^m C_i$. We claim that a vector $\vct{c}$ with at least one negative entry belongs to $C$ if and only if}
	\[
	{\color{revisionBlue}\vct{c} \in \textstyle\sum_{i: c_i < 0} C_i.}
	\]
\end{lemma}
\begin{proof}
	{\color{revisionBlue}Suppose $\vct{c} \in C$ has a decomposition $\vct{c} = \sum_{i \in N} \vct{c}^{(i)}$ where each $\vct{c}^{(i)}$ belongs to $C_i$.}
	If $N = \{ i : c_i < 0 \}$, then there is nothing to prove, so suppose there is some $k$ in $N$ with $c_k \geq 0$.
	We construct an alternative decomposition of $\vct{c}$ using only {\color{revisionBlue} cones $C_i$} with $i$ in $N \setminus \{k\}$.
	
	The construction depends on the sign of $c^{(k)}_k$. 
	If $c^{(k)}_k$ is nonnegative then the problem of removing dependence on $C_k$ {\color{revisionBlue}simple}: for $i$ in $N \setminus \{k\} $, the vectors
	\[
	\tilde{\vct{c}}^{(i)} = \vct{c}^{(i)} + \vct{c}^{(k)} / (|N| - 1)
	\]
	belong to $C_i$ {\color{revisionBlue}(since $C_i \supset \R^m_+$)}, and sum to $\vct{c}$.
	If instead  $c^{(k)}_k < 0$, then there exists some index $i \neq k$ in $N$ with $c^{(i)}_k$ positive.  
	This allows us to define the distribution $\vct{\lambda}$ with $\lambda_i = c^{(i)}_k / \sum_{j \in N \setminus \{k\} } c^{(j)}_k$ for $i \neq k$ in $N$. 
	With $\vct{\lambda}$ we construct the $|N| - 1$ vectors
	\[
	\tilde{\vct{c}}^{(i)} = \vct{c}^{(i)} + \lambda_i \vct{c}^{(k)}.
	\]
	{\color{revisionBlue}The vectors $\tilde{\vct{c}}^{(i)}$ belong to $K$ because they are a conic combiation of vectors in $K$ ($\vct{c}^{(i)}$ and $\vct{c}^{(k)}$).}
	We claim that for every $i \neq k$ in $N$, the coordinate $\tilde{c}^{(i)}_k$ is nonnegative. 
	This is certainly true when $\lambda_i = 0$, but more importantly, $\lambda_i > 0$ implies
	\begin{equation*}
	\frac{1}{\lambda_i} \tilde{c}^{(i)}_k = \frac{1}{\lambda_i}  \left( c^{(i)}_k + \lambda_i c^{(k)}_k  \right) = [\textstyle\sum_{j \in N \setminus \{k\}} c^{(j)}_k] + c^{(k)}_k = c_k \geq 0.
	\end{equation*} 
	Hence $\vct{c}$ can be expressed as the sum of vectors $\{ \tilde{\vct{c}}^{(i)}\}_{i \in N \setminus \{ k \}}$ where each vector $\tilde{\vct{c}}^{(i)}$ belongs to {\color{revisionBlue}$C_i$}.
	
	From here, update $N \leftarrow N \setminus \{k\}$.
	If $N$ contains another index $k'$ with $c_{k'} \geq 0$, 
	then repeat the above procedure to remove the unnecessary {\color{revisionBlue}cone $C_{k'}$}. 
	Naturally, this process continues until $N = \{ i : c_i < 0 \}$.
\end{proof}

\begin{lemma}\label{lem:nng_linear_span_transform}  
	Let $\vct{w}$, $\vct{v}$ be vectors in $\mathbb{R}^m$ with distinguished indices $i \neq j$ so that 
	\[
	\vct{w}_{\setminus i},~ \vct{v}_{\setminus j} \geq \vct{0} \quad \text{ and } \quad w_k + v_k < 0 \text{ for } k \text{ in } \{i,j\}.
	\]
	{\color{revisionBlue}Then there exist vectors $\vct{\hat{w}}, \vct{\hat{v}}$ in the conic hull of $\{\vct{w},\vct{v}\}$ which satisfy}
	\[
	\vct{\hat{w}} + \vct{\hat{v}} = \vct{w} + \vct{v} \quad \text{ and } \quad \hat{w}_j = \hat{v}_i = 0.
	\]
\end{lemma}
\begin{proof}
	By reindexing, take $i = 1$ and $j = 2$. Such $\vct{\hat{w}}, \vct{\hat{v}}$ exist if and only if some $\vct{\lambda}$ in $\mathbb{R}^4_+$ solves
	\begin{equation}
	\begin{bmatrix} w_1 & 0 & v_2 & 0 \\ 0 & w_2 & 0 & v_1 \\ 1 & 1 & 0 & 0 \\ 0 & 0 & 1 & 1 \end{bmatrix} \begin{bmatrix} \lambda_1 \\ \lambda_2 \\ \lambda_3 \\ \lambda_4\end{bmatrix} = \begin{bmatrix} 0 \\ 0 \\ 1 \\ 1 \end{bmatrix}. \label{eq:pairwiseNonnegSystem}
	\end{equation}
	The determinant of the matrix above is $d = w_1 v_2 - v_1 w_2$. If $w_2$ or $v_1 = 0$, then $d > 0$. If $w_2,v_1 \neq 0$, then $d > 0 \Leftrightarrow |v_2/w_2| \cdot  |w_1/v_1| > 1$. In this case we use the assumptions on $\vct{w},\vct{v}$ to establish the slightly stronger condition that $|v_2/w_2| > 1$ and $|w_1/v_1|  > 1$. In both cases we have a nonzero determinant, so there exists a unique $ \vct{\lambda} $ in $ \mathbb{R}^4$ satisfying system \eqref{eq:pairwiseNonnegSystem}. Now we need only prove that this $\vct{\lambda}$ is nonnegative.
	
	One may verify that the symbolic solution to \eqref{eq:pairwiseNonnegSystem} is 
	\begin{align*}
	& \lambda_1 = -(w_2 + v_2) v_1 / d, \quad
	\lambda_2 = (w_1 + v_1) v_2 /  d, \\
	&\lambda_3 = w_1 (w_2 + v_2) / d, \quad
	\lambda_4 = -(w_1 + v_1)w_2 / d,
	\end{align*}
	and furthermore that all numerators and denominators are nonnegative.
\end{proof}

{\color{revisionBlue}
\begin{proof}[Proof of Theorem \ref{thm:restrictSS}]
	Let $\vct{c}^\star$ be a vector in $\csage{\emat}$ with $k$ negative entries $c_1^\star,\ldots,c_k^\star$.
	It is clear that the AGE cones $\cage{\emat,i}$ satisfy the hypothesis of Lemma \ref{lem:k_neg_coeff_means_k_age_functions}, with $K = \cnns{\emat}$.
	Therefore there exists a $k$-by-$m$ matrix $\mtx{C}$ with $i^{\text{th}}$ row $\vct{c}_i \in \cage{\emat,i}$, and $\vct{c}^\star = \sum_{i=1}^k \vct{c}_i$.
	We prove the result by transforming $\mtx{C}$ into a matrix with rows $\vct{c}_i$ satisfying the required properties, using only row-sum preserving conic combinations from Lemma \ref{lem:nng_linear_span_transform}.
	
	It is clear that for any pair of distinct $i,j$, the vectors $\vct{c}_i, \vct{c}_j$ satisfy the hypothesis of Lemma \ref{lem:nng_linear_span_transform}, thus there exist $\vct{\hat{c}}_i,\vct{\hat{c}}_j$ in the conic hull of $\vct{c}_i, \vct{c}_j$ where $\hat{c}_{ij} = \hat{c}_{ji} = 0$ and $\vct{\hat{c}}_i + \vct{\hat{c}}_j = \vct{c}_i + \vct{c}_j$.
	Furthermore, this remains true if we modify $\mtx{C}$ by \textit{replacing} $(\vct{c}_i, \vct{c}_j) \leftarrow (\vct{\hat{c}}_i,\vct{\hat{c}}_j)$.
	
	We proceed algorithmically: apply Lemma \ref{lem:nng_linear_span_transform} to rows $(1,2)$, then $(1,3)$, and continuing to rows $(1,k)$.
	At each step of this process we eliminate $c_{j1} = 0$ for $j > 1$ and maintain $c_{ji} \geq 0$ for off-diagonal $c_{ji}$.
	We then apply the procedure to the second column of $\mtx{C}$, beginning with rows 2 and 3.
	Since $c_{j1} = 0$ for $j > 1$, none of the row operations introduce an additional nonzero in the first column of $\mtx{C}$, and so the first column remains zero below $c_{11}$, and the second column becomes zero below $c_{22}$.
	Following this pattern we reduce $\mtx{C}$ to have zeros on the strictly lower-triangular block in the first $k$ columns, in particular terminating with $c_{kk} = c_k^\star < 0$.
	
	The next phase is akin to back-substitution. Apply Lemma \ref{lem:nng_linear_span_transform} to rows $(k,k-1)$, then $(k,k-2)$, and continue until rows $(k, 1)$. This process zeros out the $k^{\text{th}}$ column of $\mtx{C}$ above $c_{kk}$.
	The same procedure applies with rows $(k-1,k-2)$, then $(k-1,k-3)$, through $(k-1,1)$, to zero the $(k-1)^{\text{st}}$ column of $\mtx{C}$ except for the single entry $c_{[k-1][k-1]} = c^\star_{k-1} < 0$.
	The end result of this process is that the first $k$ columns of $\mtx{C}$ comprise a diagonal matrix with entries $(c_1^\star,\ldots,c_k^\star) < \vct{0}$.
	
	The resulting matrix $\mtx{C}$ satisfies the claimed sparsity conditions.
	Since all row-operations involved conic combinations, each row of the resulting matrix $\mtx{C}$ defines a nonnegative signomial.
	The theorem follows since row $i$ of the resulting matrix has a single negative component $c_{ii} = c^
	\star_i < 0$.
\end{proof}
}  

\subsection{Proof of extreme ray characterization of the SAGE cone (Theorem \ref{thm:extreme_rays})\label{subsec:ExtremeRays}}

Because every ray in the SAGE cone (extreme or otherwise) can be written as a sum of rays in AGE cones, it suffices to characterize the extreme rays of AGE cones.
For the duration of this section we discuss the AGE cone $\cage{\emat,k}$, where $\evec_k$ is nonextremal in $\newpoly{\emat}$.

It can easily be shown that for any index $i$ in $[m]$, the ray $\{ r \vct{e}_i : r \geq 0 \}$ is extremal in $\cage{\emat,k}$. 
We call these these rays (those supported on a single coordinate) the \textit{trivial} extreme rays of the AGE cone.
The {\color{revisionBlue}work in showing} Theorem \ref{thm:extreme_rays} is to {\color{revisionBlue}prove} that all nontrivial extreme rays of the AGE cone are supported on simplicial circuits.
{\color{revisionBlue}Our proof will appeal to the following basic fact concerning polyhedral geometry, which we establish in the appendix.}

\begin{restatable}{lemma}{decomposeMixture}\label{lem:decompose_mixture}
    {\color{revisionBlue}Fix $\mtx{B} \in \R^{n \times d}$, $\vct{h} \in \R^n$, and $\Lambda = \{ \vct{\lambda} \in \Delta_d \,:\, \mtx{B}\vct{\lambda} = \vct{h}\}$.
    For any $\vct{\lambda} \in \Lambda$, there exist $\{\vct{\lambda}^{(i)}\}_{i=1}^\ell \subset \Lambda$ and $\vct{\theta} \in \Delta_\ell$ for which $\{\vct{b}_j \,:\, \lambda^{(i)}_j > 0\}$ are affinely independent, and $\vct{\lambda} = \sum_{i=1}^\ell \theta_i \vct{\lambda}^{(i)}$.}
\end{restatable}

\begin{proof}[Proof of Theorem \ref{thm:extreme_rays}]
    {\color{revisionBlue} We want to find an $\ell \in \N$ where we can decompose $\vct{c} \in \cage{\emat,k}$ as a sum of $\ell+1$ AGE vectors $\{\vct{c}^{(i)}\}_{i=1}^{\ell+1} \subset \cage{\emat,k}$, where $\{ \evec_j \,:\, c^{(i)}_j \neq 0 \}$ are simplicial circuits for $i \in [\ell]$, and $\vct{c}^{(\ell+1)}$ is elementwise nonnegative.
    Since $\vct{c}$ is an AGE vector, there is an associated $\vct{\nu} \in \R^{[m] \setminus k}_+$ for which $[\emat_{\setminus k} - \evec_{k}\vct{1}^\intercal]\vct{\nu}= \vct{0}$ and $\relent{\vct{\nu}}{e \vct{c}_{\setminus k}} \leq c_k$.
    If $\vct{\nu}$ is zero, then $\relent{\vct{\nu}}{e \vct{c}_{\setminus k}} = 0 \leq c_k$, so $\ell = 0$ and $\vct{c}^{(\ell+1)} = \vct{c}$ provides the required decomposition.
    The interesting case, of course, is when $\vct{\nu}$ is nonzero.
    Our proof proceeds by providing a mechanism to decompose $\vct{\nu}$ into a convex combination of certain vectors $\{\vct{\nu}^{(i)}\}_{i=1}^\ell$, and from there obtain suitable AGE vectors $\vct{c}^{(i)}$ from each $\vct{\nu}^{(i)}$.
    
    Given $\vct{\nu} \neq \vct{0}$, the vector $\vct{\lambda} \doteq \vct{\nu}/\vct{\nu}^\intercal \vct{1}$ belongs to the probability simplex $\Delta_{[m] \setminus k} \subset \R^{[m] \setminus k}$.
    We introduce this $\vct{\lambda}$ because $[\emat_{\setminus k} - \evec_{k}\vct{1}^\intercal]\vct{\nu}= \vct{0}$ is equivalent to $\emat_{\setminus k}\vct{\lambda}=\evec_{k}$, and the latter form is amenable to Lemma \ref{lem:decompose_mixture}.
    Apply Lemma \ref{lem:decompose_mixture} to decompose $\vct{\lambda}$ into a convex combination of vectors $\{\vct{\lambda}^{(i)}\}_{i=1}^\ell \subset \Delta_{[m]\setminus k}$ for which $\{\evec_j \,:\, j \in \supp \vct{\lambda}^{(i)}\}$ are simplicial and $\vct{\lambda}^{(i)}$ satisfy $\emat_{\setminus k}\vct{\lambda}^{(i)} = \evec_{k}$; let $\vct{\theta} \in \Delta_{\ell}$ denote the vector of convex combination coefficients for this decomposition of $\vct{\lambda}$.
    For each $\vct{\lambda}^{(i)}$, define $\vct{\nu}^{(i)} = \vct{\lambda}^{(i)}(\vct{\nu}^\intercal \vct{1})$.
    These values for $\vct{\nu}^{(i)}$ evidently satisfy $[\emat_{\setminus k} - \evec_{k}\vct{1}^\intercal]\vct{\nu}^{(i)}= \vct{0}$ and $\sum_{i=1}^\ell \theta_i \vct{\nu}^{(i)} = \vct{\nu}$.
    From these $\vct{\nu}^{(i)}$ we construct $\vct{c}^{(i)} \in \R^m$ by
    \begin{equation*}
        c^{(i)}_j = \begin{cases} (c_j / \nu_j)\nu^{(i)}_j &\text{ if } \nu_j > 0 \\
        0 & \text{ otherwise } \end{cases}
        \qquad \text{ for all } j \neq k,\label{eq:simplicialAgeVectors}
    \end{equation*}
    and for $j = k$ we take $c^{(i)}_k = \relent{\vct{\nu}^{(i)}}{ e\vct{c}^{(i)}_{\setminus k}}$.
    
    By construction these $\vct{c}^{(i)}$ belong to $\cage{\emat,k}$, and $\{ \evec_j \,:\, c^{(i)}_j \neq 0 \}$ comprise simplicial circuits.
    We now take a componentwise approach to showing $\sum_{i=1}^\ell \theta_i \vct{c}^{(i)} \leq \vct{c}$.
    For indices $j \neq k$ with $\nu_j > 0$, the inequality actually holds with \textit{equality} $\sum_{i=1}^\ell \theta_i c^{(i)}_j = (c_j / \nu_j) (\sum_{i=1}^\ell \theta_i \nu^{(i)}_j) = c_j$.
    Now we turn to showing $\sum_{i=1}^\ell \theta_i c^{(i)}_k \leq c_k$; we specifically claim that
    \begin{equation}
    \sum_{i=1}^\ell \theta_i c^{(i)}_k = \sum_{i=1}^\ell \theta_i \relent{\vct{\nu}^{(i)}}{e\vct{c}^{(i)}_{\setminus k}} = \relent{\vct{\nu}}{e\vct{c}_{\setminus k}} \leq c_k. \label{eq:ageDecompOnNegCoordinate}
    \end{equation}
    For the three relations in display \eqref{eq:ageDecompOnNegCoordinate}, the first holds from the definitions of $c^{(i)}_k$, and the last holds from our assumptions on $(\vct{c},\vct{\nu})$, so only the second equality needs explaining.
    For this we use the fact that definitions of $c^{(i)}_j$ relative to $\nu^{(i)}_j$ preserve ratios with $c_j$ relative to $\nu_j$, i.e.
    \begin{equation}
    \relent{\vct{\nu}^{(i)}}{e\vct{c}^{(i)}_{\setminus k}}
    = \sum_{j \neq k} \nu^{(i)}_j \ln\left(\frac{\nu^{(i)}_j}{e c^{(i)}_j}\right) = \sum_{j \neq k} \nu^{(i)}_j \ln\left(\frac{\nu_j}{e c_j}\right). \label{eq:ageDecompNegCoordinate_intermediate}
    \end{equation}
    One may then prove the middle equality in display \eqref{eq:ageDecompOnNegCoordinate} by summing $\theta_i \relent{\vct{\nu}^{(i)}}{\vct{c}_{\setminus k}}$ over $i$, applying the identity in equation \eqref{eq:ageDecompNegCoordinate_intermediate}, and then interchanging the sums over $i$ and $j$.
    Formally,
    \[
    \sum_{i=1}^\ell \theta_i \relent{\vct{\nu}^{(i)}}{e\vct{c}^{(i)}_{\setminus k}} =
    \sum_{j \neq k}\log\left(\frac{\nu_j}{ec_j}\right)\underbrace{\left(\sum_{i=1}^\ell\theta_i \nu^{(i)}_j \right)}_{=\nu_j} = \sum_{j \neq k} \nu_j \log\left(\frac{\nu_j}{ec_j}\right) = \relent{\vct{\nu}}{e\vct{c}_{\setminus k}}.
    \]
    
    We have effectively established the claim of the theorem.
    To find a decomposition of the form desired at the beginning of this proof, one rescales $\vct{c}^{(i)} \leftarrow \theta_i \vct{c}^{(i)}$ and sets $\vct{c}^{(\ell+1)} = \vct{c} - \sum_{i=1}^\ell \vct{c}^{(i)}$.}
\end{proof}

\section{The Role of Newton Polytopes in SAGE Signomials}\label{sec:structureSAGEandNNG}

This section begins by introducing two theorems (Theorems \ref{thm:linear_indep_and_all_nonext_neg_is_exact} and \ref{thm:needNonnegOfFaces}) concerning SAGE representability versus signomial nonnegativity. 
These theorems are then combined to obtain a third theorem (Theorem \ref{thm:SageEqNng}), which provides the most general yet-known conditions for when the SAGE and nonnegativity cones coincide.
The proofs of Theorems \ref{thm:linear_indep_and_all_nonext_neg_is_exact} and \ref{thm:needNonnegOfFaces}  are contained in Sections  \ref{subsec:proofOfThmLIANNE} and \ref{subsec:proofOfThmNeedNngOfFaces}.
Applications of Theorem \ref{thm:linear_indep_and_all_nonext_neg_is_exact}
are given in Section \ref{sec:SigOpt}.
{\color{revisionBlue}Section \ref{subsec:finiteError} uses a distinct proof strategy (nevertheless Newton-polytope based) to determine a condition on when SAGE can recognize signomials which are bounded below.}

\subsection{When SAGE recovers the nonnegativity cone}\label{subsec:newpolyMainResults}

The following theorem is the first instance beyond AGE functions when SAGE-representability is known to be equivalent to nonnegativity.

\begin{restatable}[]{theorem}{LIANNE}\label{thm:linear_indep_and_all_nonext_neg_is_exact}
	Suppose $\ext\newpoly{\emat}$ is simplicial, and that $\vct{c}$ has $c_i \leq 0$ whenever $\evec_i$ is nonextremal.
	Then $\vct{c} $ belongs to $\csage{\emat}$ iff $\vct{c}$ belongs to  $\cnns{\emat}$.
\end{restatable}
\textcolor{revisionBlue}{Our proof of the theorem (Section \ref{subsec:proofOfThmLIANNE}) uses convex duality in a central way, and provides intuition for why the theorem's assumptions are needed.
Section \ref{sec:constructCounterExAndDualChar} provides counter-examples to relaxations of Theorem \ref{thm:linear_indep_and_all_nonext_neg_is_exact} obtained through weaker hypothesis.}

This section's next theorem (proven in Section \ref{subsec:proofOfThmNeedNngOfFaces}) concerns conditions on $\emat$ for when the SAGE and nonnegativity cones can be expressed as a Cartesian product of simpler sets. 
{\color{revisionBlue}To aid in exposition we introduce a definition: a matrix $\emat$ can be \textit{partitioned into $k$ faces} if by a permutation of its columns it can be written as a concatenation $\emat =  [\emat^{(1)},\ldots,\emat^{(k)}]$, where $\emat^{(i)}$ are submatrices of $\emat$ and $\{\newpoly{\emat^{(i)}}\}_{i=1}^k$ are mutually disjoint faces of $\newpoly{\emat}$.
}

\begin{theorem}	
	If $\{\emat^{(i)} \}_{i=1}^k$ are matrices partitioning {\color{revisionBlue}the block matrix} $\emat = [\emat^{(1)},\ldots,\emat^{(k)}]$, then
	$\cnns{\emat}= \oplus_{i=1}^k \cnns{\emat^{(i)}}$ and $\csage{\emat} = \oplus_{i=1}^k \csage{\emat^{(i)}}$. \label{thm:needNonnegOfFaces}
\end{theorem}

{\color{revisionBlue}The following figure illustrates partitioning a matrix $\emat$ where $\ext \newpoly{\emat}$ are the vertices of the truncated icosahedron, and nonextremal terms (marked in red) lay in the relative interiors of certain pentagonal faces.}
\begin{center}
	\includegraphics[width=0.4\textwidth]{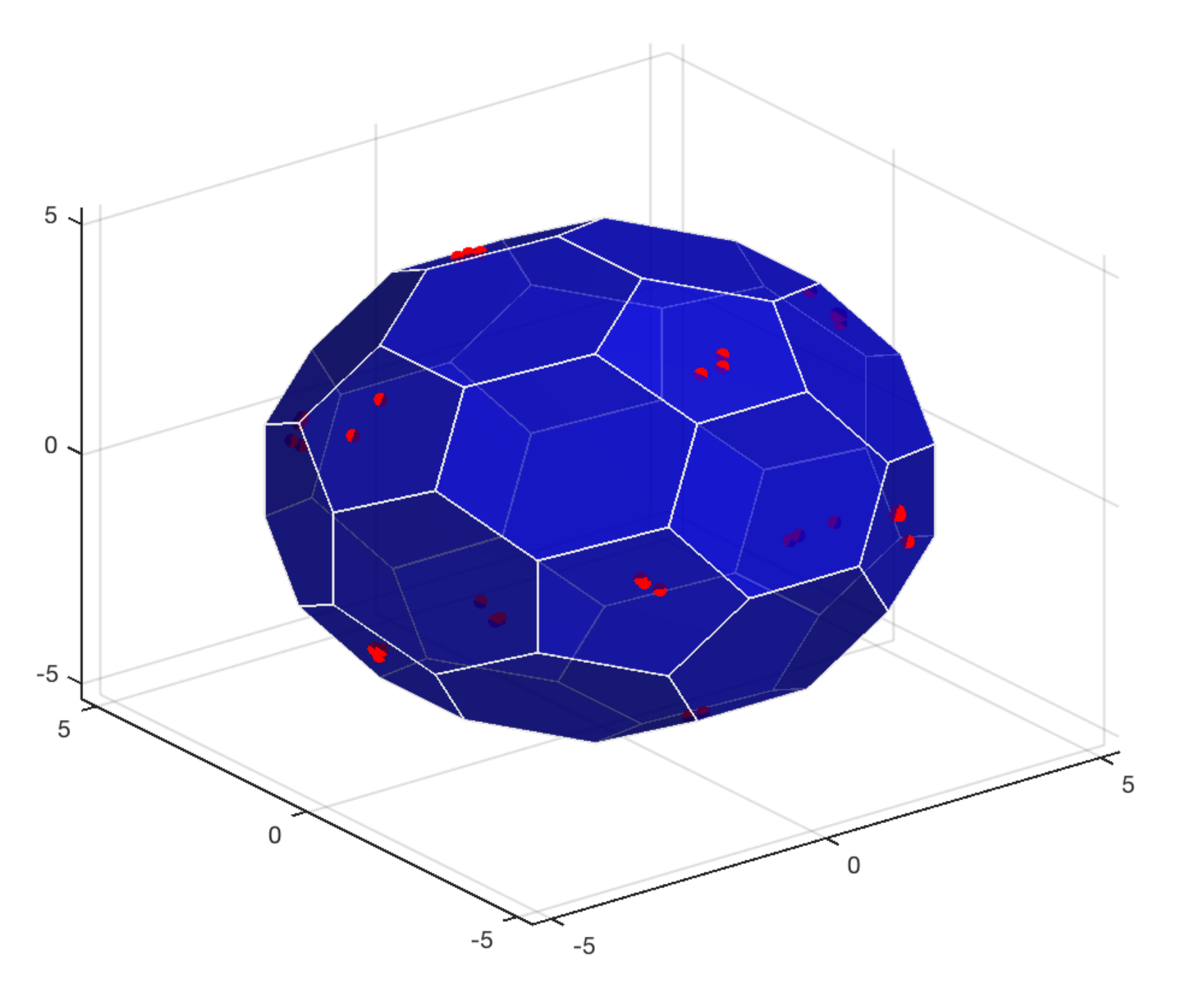}
	\hspace{0.5in}\includegraphics[width=0.4\textwidth]{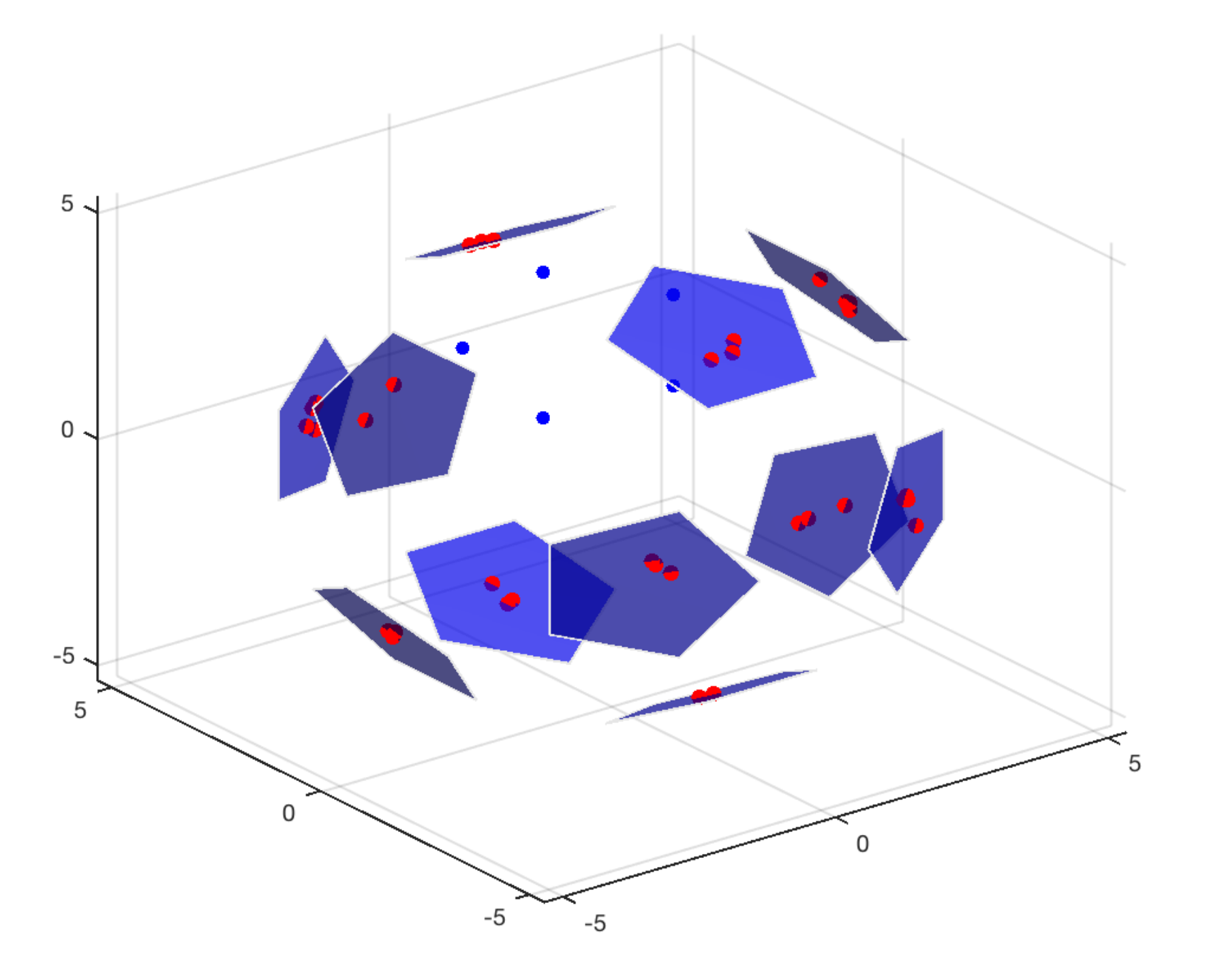}
\end{center}
Note that every matrix $\emat$ admits the trivial partition with $k = 1$.
In fact, a natural regularity condition (one that we consider in Section \ref{sec:constructCounterExAndDualChar}) would be that $\emat$ \textit{only} admits the trivial partition.
Regularity conditions aside, Theorems \ref{thm:linear_indep_and_all_nonext_neg_is_exact} and  \ref{thm:needNonnegOfFaces} can be combined with known properties of AGE functions to establish new conditions for when the {\color{revisionBlue}SAGE and nonnegativity cones coincide}. 

\begin{restatable}[]{theorem}{SageEqNng}\label{thm:SageEqNng}
	Suppose $\emat$ can be partitioned into faces where (1)
	simplicial {\color{revisionBlue}faces contain} at most two nonextremal exponents,
	and (2) all other faces contain at most one nonextremal exponent. Then $\csage{\emat} = \cnns{\emat}$.
\end{restatable}

\begin{proof}
	Let $\emat$ satisfy the assumptions of Theorem \ref{thm:SageEqNng} with associated faces $\{F_i\}_{i=1}^k$ and column blocks $\emat^{(i)}$, and fix $\vct{c} $ in $ \cnns{\emat}$. For $i$ in $[k]$, define the vector $\vct{c}^{(i)}$ so that $\vct{c} = \oplus_{i=1}^k \vct{c}^{(i)}$. By Theorem \ref{thm:needNonnegOfFaces},  the condition $\csage{\emat} =\cnns{\emat}$ holds if and only if $\csage{\emat^{(i)}}=\cnns{\emat^{(i)}}$ for all $i$ in $[k]$. Because we assumed that $\vct{c} $ belongs to $\cnns{\emat}$ it suffices to show that each $\vct{c}^{(i)} $ belongs to $\csage{\emat^{(i)}}$.
	
	{\color{revisionBlue}Per Proposition \ref{prop:sage_def_wrt_exts}, any vector $\vct{c}^{(i)} \in \cnns{\emat^{(i)}}$ cannot have a negative entry $c^{(i)}_j$ when $\evec^{(i)}_j$ is extremal in $\newpoly{\emat^{(i)}}$.
	By assumption, $\emat^{(i)}$ has at most two nonextremal terms, and so $\vct{c}^{(i)} \in \cnns{\emat^{(i)}}$ can have at most two negative entries.}
	If $\vct{c}^{(i)}$ has at most one negative entry, then $\vct{c}^{(i)}$ is \textcolor{revisionBlue}{an AGE vector}.
	If on the other hand $\vct{c}^{(i)}$ has two negative entries {\color{revisionBlue}$c_j^{(i)}$}, then both of these entries must correspond to nonextremal $\evec_j$, and $F_i$ must be simplicial. 
	This allows us to invoke Theorem \ref{thm:linear_indep_and_all_nonext_neg_is_exact} on $\vct{c}^{(i)}$ to conclude $\vct{c}^{(i)} \in \csage{\emat^{(i)}}$. 
	The result follows.
\end{proof}

\subsection{Simplicial sign patterns for SAGE versus nonnegativity (Theorem \ref{thm:linear_indep_and_all_nonext_neg_is_exact})}\label{subsec:proofOfThmLIANNE}

The proof of Theorem \ref{thm:linear_indep_and_all_nonext_neg_is_exact} begins by exploiting two key facts about signomials and SAGE relaxations: (1) that $\csage{\emat}$ and $\cnns{\emat}$ are invariant under translation of the exponent set $\emat$, and (2) that strong duality \textit{always} holds when computing $f_{\mathsf{SAGE}}$. 
These properties allow us to reduce the problem of checking SAGE decomposability to the problem of exactness of a convex relaxation for a signomial optimization problem.

\begin{proof}[Proof of Theorem \ref{thm:linear_indep_and_all_nonext_neg_is_exact}]
	Begin by translating $\emat $ in $ \mathbb{R}^{n \times m}$ to $\emat \leftarrow \emat - \evec_j \mathbf{1}^\intercal$ where $\evec_j$ is an arbitrary extremal element of $\newpoly{\emat}$.
	{\color{revisionBlue}Next}, permute the columns of $\emat$ so that $\evec_1 = \vct{0}$.
	Fix $\vct{c} $ in $\cnns{\emat}$ and define $f = \Sig(\emat, \vct{c})$ so that $f^\star \geq 0$.
	We show that $f_{\mathsf{SAGE}}= f^\star$, thereby establishing {\color{revisionBlue}$\vct{c} \in \csage{\emat}$}.
	
	Let $N = \{ i : c_i \leq 0 \}$ {\color{revisionBlue}and $E = [m] \setminus N$}; apply Corollary \ref{cor:sagerelax_restrictss_primal} {\color{revisionBlue}with Proposition \ref{prop:strong_duality}} to obtain
	\begin{align}
	f_{\mathsf{SAGE}} =~ \inf~&~\vct{c}^\intercal \vct{v} \label{prob:reduced_dual_Lemma2.3}\\
	\text{s.t.} ~& ~\vct{v} \text{ in } \mathbb{R}^{m}_{++} \text{ has } v_1 = 1, {\color{revisionBlue}\text{ and there exist } \{\vct{\mu}_i\}_{i \in N \cup \{1\}} \subset \R^n} \text{ with }\nonumber  \\
	&~v_i \ln(v_i/v_j) \leq (\evec_i- \evec_j)^\intercal \vct{\mu}_i \text{ for }  j \text{ in } {\color{revisionBlue}E} \text{ and } i \text{ in } N \cup \{1\} \nonumber .
	\end{align}
	In order to show $f_{\mathsf{SAGE}}= f^\star$, we reformulate \eqref{prob:reduced_dual_Lemma2.3} as the problem of computing $f^\star$ by appropriate changes of variables and constraints.
	
	We begin with a change of constraints.
	By the assumption that $c_i \leq 0$ for all nonextremal $\evec_i$, the set $E$ satisfies $\{ \evec_i \}_{i \in E} \subset \ext \newpoly{\emat}$.
	Combine {\color{revisionBlue}this with extremality of $\vct{0}  = \evec_1$} and the assumption that $\ext \newpoly{\emat}$ is simplicial to conclude that $\{ \evec_i : i \text{ in } E \setminus \{1 \}\}$ are  linearly independent.
	The linear independence of these vectors ensures that for fixed $\vct{v}$ we can always choose $\vct{\mu}_1$ to satisfy the following constraints \textit{with equality}
	\begin{equation*}
	v_1 \ln(v_1/v_j) \leq (\evec_1- \evec_j)^\intercal \vct{\mu}_1 \text{ for all }  j \text{ in } {\color{revisionBlue}E}.
	\end{equation*}
	Therefore we can equivalently reformulate $f_{\mathsf{SAGE}}$ as 
	\begin{align*}
	f_{\mathsf{SAGE}} = \inf ~&~\vct{c}^\intercal \vct{v} \\
	\text{s.t.} ~&~\vct{v} \text{ in } \mathbb{R}^{m}_{++} \text{ has } v_1 = 1,{\color{revisionBlue}\text{ and there exist } \{\vct{\mu}_i\}_{i \in N \cup \{1\}} \subset \R^n}  \\
	&\text{ with } \ln (v_j) = \evec_j^\intercal \vct{\mu}_1 \text{ for all } j \text{ in } E, \text{ and } \\
	&~v_i \ln(v_i/v_j) \leq (\evec_i- \evec_j)^\intercal \vct{\mu}_i  \text{ for } j \text{ in } E,~ i \text{ in } N.
	\end{align*}
	
	{\color{revisionBlue}Next we rewrite the constraint} $v_i \ln(v_i/v_j) \leq (\evec_i- \evec_j)^\intercal \vct{\mu}_i$ as $\ln(v_i)-\ln(v_j) \leq (\evec_i- \evec_j)^\intercal \vct{\mu}_i $ by absorbing $v_i$ into $\vct{\mu}_i$.
	If we also substitute the expression for $\ln(v_j)$ given by the equality constraints, then the inequality constraints become
	\begin{equation}
	\ln(v_i) \leq  \evec_i^\intercal \vct{\mu}_i +  \evec_j^\intercal(\vct{\mu}_1 - \vct{\mu}_i)  \text{ for all  }  j \text{ in } E,~i \text{ in } N. \label{eq:inequalityConstrsAfterDivideOutAndSubstituteVj}
	\end{equation}
	{\color{revisionBlue}We now show that for} every $i$ in $N$, the choice $\vct{\mu}_i = \vct{\mu}_1$ makes these inequality constraints as loose as possible.
	
	Towards this end, define $\psi_i(\vct{x}) =  \evec_i^\intercal  \vct{x} + \min_{j \in E} \{ \evec_j^\intercal ( \vct{\mu}_1 - \vct{x})  \}$; note that for fixed $i$ and $\vct{\mu}_i$, the number $\psi_i(\vct{\mu}_i)$ is the minimum over all $|E|$ right hand sides in \eqref{eq:inequalityConstrsAfterDivideOutAndSubstituteVj}.
	It is easy to verify that $\psi_i$ is concave, and because of this 
	we know that $\psi_i$ is maximized at $\vct{x}^\star$ if and only if $\vct{0} \in (\partial \psi_i)(\vct{x}^\star)$.
	Standard subgradient calculus tells us that $(\partial \psi_i)(\vct{x})$ is precisely the convex hull of vectors $\evec_i - \evec_k$ where $k$ is an index at which the minimum (over $j \in E$) is obtained.
	Therefore $(\partial \psi_i)(\vct{\mu}_1) = \conv\{ \evec_i - \evec_j : j \text{ in } E \}$, and this set must contain the zero vector (unless perhaps $c_i = 0$, in which case the constraints on $v_i$ are inconsequential).
	Hence $\max_{\vct{x} \in \mathbb{R}^n} \{\psi_i(\vct{x}) \} =  \evec_i^\intercal \vct{\mu}_1$, 
	and so inequality constraints \eqref{eq:inequalityConstrsAfterDivideOutAndSubstituteVj} reduce to
	\begin{equation}
	\ln(v_i) \leq  \evec_i^\intercal \vct{\mu}_1  \text{ for all }  i  \text{ in } N. \label{eq:constraintsSimplifiedExactnessLemma}
	\end{equation}
	
	Since the objective $\vct{c}^\intercal \vct{v}$ is decreasing in $v_i$ for $i $ in $N$, 
	we can actually take the constraints in \eqref{eq:constraintsSimplifiedExactnessLemma} to be binding. 
	We established much earlier that $v_i = \exp \evec_i^\intercal \vct{\mu}_1$ for $i $ in $E$. 
	Taking these together we see $v_i = \exp \evec_i^\intercal \vct{\mu}_1$ for all $i $, and so 
	\begin{equation}
    	f_{\mathsf{SAGE}} = \inf\{ ~\textstyle\sum_{i=1}^m c_i \exp  \evec_i^\intercal \vct{\mu}_1 ~:~ \vct{\mu}_1 \text{ in } \mathbb{R}^n  \} = f^\star
	\end{equation}
	as required. 
\end{proof}

Let us now recap how the assumptions of Theorem \ref{thm:linear_indep_and_all_nonext_neg_is_exact} were used at various stages in the proof.
For one thing, all discussion up to and including the statement of Problem \eqref{prob:reduced_dual_Lemma2.3} was fully general; the expression for $f_{\mathsf{SAGE}}$ used none of the assumptions of the theorem.
The next step was to use linear independence of nonzero extreme points to allow us to satisfy  $v_1 \ln(v_1/v_j) \leq (\evec_1 - \evec_j)^\intercal \vct{\mu}_1$ with equality.
The reader can verify that if we did not have linear independence but we \textit{were} told that those constraints were binding at the optimal $\vct{v}^\star$, then we would still have $f_{\mathsf{SAGE}} = f^\star$ under the stated sign pattern assumption on $\vct{c}$.
{\color{revisionBlue}Note how the sign pattern assumption on $\vct{c}$ was only really used to replace $\ln(v_i) \leq \evec_i^\intercal \vct{\mu}_1$ from \eqref{eq:constraintsSimplifiedExactnessLemma} by $\ln(v_i) = \evec_i^\intercal \vct{\mu}_i$.}

\subsection{Proof of the partitioning theorem (Theorem \ref{thm:needNonnegOfFaces})}\label{subsec:proofOfThmNeedNngOfFaces}

\textcolor{revisionBlue}{The following} lemma adapts claim (iv) from Theorem 3.6 of Reznick \cite{Reznick1989} to signomials.
Because the lemma is important for our subsequent theorems, {\color{revisionBlue}the appendix contains a more} complete proof than can be found in Reznick's \cite{Reznick1989}. 
As a matter of notation: for any $F \subset \{\evec_i\}_{i=1}^m$, write $\Sig_F(\emat, \vct{c})$ to mean the signomial with exponents $\evec_i$ in $F$ and corresponding coefficients $c_i$.

\begin{restatable}{lemma}{partitionLemma}\label{lem:ifgisface}
	If $F$ is a face of $\newpoly{\emat}$ then $\Sig_F(\emat,\vct{c})^\star < 0$ implies $\Sig(\emat,\vct{c})^\star < 0$.
\end{restatable}

\begin{proof}[Proof of Theorem \ref{thm:needNonnegOfFaces}]
	Let $\emat$ have partition $\emat =  [\emat^{(1)}, \ldots, \emat^{(k)}]$, {\color{revisionBlue}where the submatrices $\emat^{(i)}$ have sizes $n \times m_i$ and $\sum_{i=1}^k m_i = m$}.
	It is clear from the definition of the SAGE cone that $\csage{\emat} = \oplus_{i=1}^k \csage{\emat^{(i)}}$. 
	The bulk of this proof is to show that $\cnns{\emat}$ admits the same decomposition.
	
	Let $f = \Sig(\emat,\vct{c})$ for some $\vct{c}$ in $\mathbb{R}^m$.
	The vector $\vct{c}$ is naturally decomposed into $\vct{c} = \oplus_{i=1}^k \vct{c}^{(i)}$ {\color{revisionBlue}where $\vct{c}^{(i)} \in \R^{m_i}$ align with $\emat^{(i)}$.}
	For each $i$ in $[k]$ define $f_i = \Sig(\emat^{(i)}, \vct{c}^{(i)})$ so that $f = \sum_{i=1}^k f_i$. 
	If any $f_i^\star$ is negative, then Lemma \ref{lem:ifgisface} tells us that $f^\star$ must also be negative.
	Meanwhile if all $f_i^\star$ are nonnegative, then the same must be true of $f^\star \geq \sum_{i=1}^k f_i^\star$. The result follows. 
\end{proof}

\subsection{Corollaries for signomial programming}\label{sec:SigOpt}

Signomial minimization is naturally related via duality to checking signomial nonnegativity.
Thus we build on groundwork laid in Sections \ref{sec:SAGEcontributions} and \ref{sec:structureSAGEandNNG} to obtain consequences for signomial minimization.

\begin{corollary}\label{cor:simplicial_optimization_zero_anywhere}
	Assume $\newpoly{\emat}$ is simplicial \textcolor{revisionBlue}{with $\evec_1 = \vct{0}$}, and that nonzero nonextremal $\evec_i$ have $c_i \leq 0$.
	Then either $f_{\mathsf{SAGE}} = f^\star$, or $f^\star \in (f_{\mathsf{SAGE}}, ~c_1)$.
\end{corollary}
\begin{proof}
	It suffices to show that $f_{\mathsf{SAGE}} < f^\star$ implies $f^\star < c_1$. This follows as the contrapositive of the following statement: ``If $f^\star \geq c_1$, then by Theorem \ref{thm:linear_indep_and_all_nonext_neg_is_exact} the nonnegative signomial $f - f^\star$ is SAGE, which in turn ensures $f_{\mathsf{SAGE}} = f^\star$.''
\end{proof}

Now we consider constrained signomial programs. 
Starting with problem data $(f,g)$ where  $f = \Sig(\emat,\vct{c})$, $g_j = \Sig(\emat,\vct{g}_j)$ for $j$ in $[k]$, \textcolor{revisionBlue}{and $\evec_1 = \vct{0}$}, consider the problem of computing 
\begin{equation}
    (f,g)^\star \doteq \inf\{ f(\vct{x}) ~:~ \vct{x} \text{ in } \mathbb{R}^n \text{ satisfies } g(\vct{x}) \geq \vct{0} \}.\label{eq:constr_opt_standard_form}
\end{equation}
It is evident\footnote{See Section 3.4 of \cite{CS16}.} that we can relax the problem to that of
\begin{align*}
    (f,g)_{\mathsf{SAGE}} \doteq& \inf\{~ \vct{c}^\intercal \vct{v} ~: ~ \vct{v} \text{ in } \csage{\emat}^\dagger \text{ satisfies } v_1 = 1 \text{ and } \mtx{G}^\intercal \vct{v} \geq \vct{0}  \} \leq (f,g)^\star 
\end{align*}
where $\mtx{G}$ is the $m \times k$ matrix whose columns are the $\vct{g}_j$.
\begin{corollary}\label{cor:constrained_case_exact}
	Suppose $\newpoly{\emat}$ is simplicial with vertex $\evec_1 =  \vct{0}$, and that when $\evec_i$ is nonextremal we have
	{\color{revisionBlue}$\mathrm{(i)}$ $\vct{c}^\intercal \vct{v}$ is decreasing in $v_i$, and $\mathrm{(ii)}$ each $\vct{g}_j^\intercal \vct{v}$ is increasing in $v_i$.} Then  $(f,g)_{\mathsf{SAGE}} = (f,g)^\star$.
\end{corollary}

\begin{proof}[Proof sketch]
	The claim that $\constrsagedual{f}{g} = (f,g)^\star$ can be established by a change-of-variables and change-of-constraints argument of the same kind used in the proof of Theorem \ref{thm:linear_indep_and_all_nonext_neg_is_exact}. 
	
	Suffice it to say that rather than using {\color{revisionBlue}Corollary \ref{cor:sagerelax_restrictss_primal}} to justify removing constraints from the dual without loss of generality, one can simply throw out those constraints to obtain some $(f,g)'$ with $(f,g)' \leq (f,g)_{\mathsf{SAGE}}$. One then shows $(f,g)' = (f,g)^\star$ to sandwich $(f,g)^\star \leq (f,g)' \leq \constrsagedual{f}{g} \leq (f,g)^\star$.
\end{proof}

\subsection{Finite error in SAGE relaxations}\label{subsec:finiteError}

This section's final theorem directly considers SAGE as a relaxation scheme for signomial minimization.
It exploits the primal formulation for $f_{\mathsf{SAGE}}$ to establish sufficient conditions under which SAGE relaxations can only exhibit finite error.
\begin{theorem}
	Suppose {\color{revisionBlue}$\vct{0} \in \newpoly{\emat}$ and} there exists an $\epsilon > 0$ so that $(1+\epsilon)\evec_j $ belongs to $ \newpoly{\emat}$ for all nonextremal $\evec_j$. Then $f = \Sig(\emat,\vct{c})$ is bounded below iff $f_{\mathsf{SAGE}}$ is finite.\label{thm:fsageisbounded1pluseps}
\end{theorem}
The requirements Theorem \ref{thm:fsageisbounded1pluseps} imposes on the Newton polytope are significantly weaker than those found elsewhere in this work.
{\color{revisionBlue}Theorem \ref{thm:fsageisbounded1pluseps} is especially notable as} we do not know of analogous theorems in the literature on SOS relaxations for polynomial optimization.

\begin{proof}[Proof of Theorem \ref{thm:fsageisbounded1pluseps}]
	Let $f = \Sig(\emat,\vct{c})$ have $\evec_1 = \vct{0}$ and $f^\star > -\infty$.
	We may assume without loss of generality that $c_1 = 0$.
	Use $E = \{ i : \evec_i \text{ nonzero, extremal}\}$ to denote indices of extremal exponents of $f$, excluding the possibly-extremal exponent $\evec_1 = \vct{0}$.
	The desired claim holds if there exists a positive constant $\gamma$ so that the translate $f_{\gamma} = f + \gamma$ is SAGE. 
	
	Define $\vct{\hat{c}} = \vct{c} + \gamma \vct{e}_1$ as the coefficient vector of $f_{\gamma}$.
	Because $f^\star > - \infty$ we have $c_i = \hat{c}_i \geq 0$ for every $ i $ in $ E$ (Proposition \ref{prop:sage_def_wrt_exts}).
	Let $N$ denote the set of indices $i$ for which $\hat{c}_i < 0$. 
	For each such index $i {\color{revisionBlue}\in N}$ we define the vector $\vct{\hat{c}}^{(i)} $ in $\mathbb{R}^m$ by
	\begin{equation*}
	\hat{c}^{(i)}_j = 
	\begin{cases}  
	\hat{c}_i & \text{ if } j = i \\
	\hat{c}_j / |N| &\text{ if } j \in {\color{revisionBlue} [m]\setminus  N} \\
	0 & \text{ if } j \in {\color{revisionBlue} N\setminus \{i\}}
	\end{cases}.
	\end{equation*}
	Certainly, $\sum_{i \in N} \vct{\hat{c}}^{(i)} = \vct{\hat{c}}$ and $\vct{\hat{c}}^{(i)}_{\setminus i} $ is nonnegative (in particular $\hat{c}^{(i)}_1 = \gamma / |N|$ is positive).
	
	Now we build the vectors $\vct{\nu}^{(i)}$ for $i$ in $N$.	
	Because $N$ is contained within $[m] \setminus E$, we have that each $i \text{ in } N$ satisfies $(1+\epsilon)\evec_i $ in $\newpoly{\emat}$ for some positive $\epsilon$.
	Therefore for $i $ in $N$, the vector $\evec_i$ is expressible as a convex combination of extremal exponents and the zero vector.
	Let $(\lambda^{(i)}_j)_{j \in E \cup \{1\}}$ be positive convex combination coefficients so that $\evec_i = \sum_{j \in E \cup \{1\}} \lambda^{(i)}_j \evec_j$.
	
	Now define the vector {\color{revisionBlue}$\vct{\nu}^{(i)} $ in $ \mathbb{R}^{[m]\setminus i}$ by $\nu^{(i)}_j = \lambda^{(i)}_j$ for $j $ in $ E \cup \{1\}$, and $\nu^{(i)}_j = 0$ for all remaining indices.}
	Each $\vct{\nu}^{(i)}$ {\color{revisionBlue}is nonnegative, satisfies $\nu^{(i)}_1 > 0$, and belongs to the kernel of $[\emat_{\setminus i}-\evec_i\vct{1}^\intercal]$.}
	Because $\nu^{(i)}_1$ is positive, the quantity {\color{revisionBlue}$D(\vct{\nu}^{(i)}, \vct{\hat{c}}^{(i)}_{\setminus i}) $} can be made to diverge to $-\infty$ by sending $\gamma$ to $\infty$. It follows that there exists a sufficiently large $M$ so that $\gamma \geq M$ implies
	{\color{revisionBlue}
	\begin{equation}
	\relent{\vct{\nu}^{(i)}}{e \vct{\hat{c}}^{(i)}_{\setminus i}} - \hat{c}_i \leq 0 \quad \text{ for all } i \text{ in } N. \label{eq:1plusepsboundedness2}
	\end{equation}}
	{\color{revisionBlue}Hence for} sufficiently large $\gamma$, we have $\vct{\hat{c}}^{(i)}$ in $\cage{\emat,i}$ for all $i $ in $ N$-- and the result follows.  
\end{proof}

\section{Certifying Global Nonnegativity of Polynomials}\label{sec:PolyNNG}

{\color{revisionBlue}Throughout this section we write $p = \Poly(\emat,\vct{c})$ to mean that $p$ takes values $p(\vct{x}) = \sum_{i=1}^m c_i \vct{x}^{\evec_i}$.
We refer to polynomials in this way to reflect our interest in \textit{sparse polynomials}.
Vectors $\evec_i$ are sometimes called \textit{terms}, where a term is \textit{even} if $\evec_i$ belongs to $(2\N)^n$.}
To an $n$-by-$m$ matrix of nonnegative integers $\emat$, we associate the sparse nonnegativity cone
\[
    \cnnp{\emat} \doteq \{ \vct{c} : \Poly(\emat, \vct{c})(\vct{x}) \geq 0 \text{ for all } \vct{x} \text{ in } \mathbb{R}^n  \}.
\]
{\color{revisionBlue}Beginning with Section \ref{subsec:sagepoly} we introduce \textit{polynomial SAGE certificates}.
We shall see that polynomial SAGE certificates offer a tractable avenue for optimizing over a subset of $\cnnp{\emat}$, where the complexity depends on $\emat$ exclusively  through the dimensions $n$ and $m$.}

{\color{revisionBlue}Section \ref{subsec:polyCorollaries} demonstrates how our study of SAGE signomials yields several corollaries in this new polynomial setting.
Perhaps most prominently, Section \ref{subsec:polyCorollaries} implies that a polynomial admits a SAGE certificate if and only if it admits a SONC certificate.
The qualitative relationship between SAGE and SONC as proof systems is explained in Section \ref{subsec:amgm}, and Section \ref{subsec:sonc} addresses how some of our corollaries compare to earlier results in the SONC literature.}

{\color{revisionBlue}In Section \ref{subsec:sos} we compare polynomial SAGE certificates to the widely-studied Sums-of-Squares certificates.
We conclude with Section \ref{subsec:polyhier}, which outlines how to use SAGE polynomials to obtain a hierarchy for constrained polynomial optimization.}

\subsection{Signomial representatives and polynomial SAGE certificates}\label{subsec:sagepoly}
\textcolor{revisionBlue}{
To a polynomial $p = \Poly(\emat,\vct{c})$ we associate the \textit{signomial representative} $q = \Sig(\emat, \vct{\hat{c}})$
with
\begin{equation}
    \hat{c}_i = \begin{cases}
    ~~~c_i &\text{ if } \evec_i \text{ is even } \\
    -|c_i| &\text{ otherwise }
    \end{cases}. \label{eq:sig_rep_choose_c}
\end{equation}}
\hspace{-0.8ex}{\color{revisionBlue}By a termwise argument, we have that if the signomial $q$ is nonnegative on $\R^n$, then the polynomial $p$ must also be nonnegative on $\R^n$.
Moving from a polynomial to its signomial representative often entails some loss of generality.
For example, the univariate polynomial $p(x) = 1 + x - x^3 + x^4$ never has both ``$+x < 0$'' and ``$-x^3 < 0$,'' and yet the inner terms appearing in the signomial representative $q(y) = 1 - \exp(y) - \exp(3y) + \exp(4y)$ are both negative.

There is a natural condition $\emat$ and the sign pattern of $\vct{c}$ where passing to the signomial representative is at no loss of generality.
Specifically,
if there exists a point $\vct{x}_0 \in (\R\setminus \{0\})^n$ where $c_i \vct{x}_0^{\evec_i} \leq 0$ for all $\evec_i \not\in (2\N)^n$,
then $\Poly(\emat,\vct{c})$ is nonnegative if and only if its signomial representative is nonnegative.
We call such polynomials \textit{orthant-dominated}.
Checking if a polynomial is orthant-dominated is a simple task. 
Given $\emat$ and $\vct{c}$, define $\vct{b}$ by $b_i = 0$ if $c_i \leq 0$ or $\evec_i$ is even, and $b_i = 1$ if otherwise.
Then assuming every $c_i \neq 0$, the polynomial $\Poly(\emat,\vct{c})$ is orthant-dominated if and only if the system $\emat^\intercal \vct{s} = \vct{b} \pmod{2}$ has a solution over $\vct{s} \in \mathbb{F}^n_2$.

In what should feel natural, we call $p = \Poly(\emat,\vct{c})$ a \textit{SAGE polynomial} if its signomial representative $q = \Sig(\emat,\vct{\hat{c}})$ is a SAGE signomial.
Subsequently, we define a \textit{polynomial SAGE certificate} for $p= \Poly(\emat,\vct{c})$ as a set of signomial AGE certificates $\{(\vct{\hat{c}}^{(i)},\vct{\nu}^{(i)})\}_{i=1}^m$ where $\vct{\hat{c}} \doteq \sum_{i=1}^m \vct{\hat{c}}^{(i)}$ defines the signomial representative for $p$.
Because the signomial SAGE cone contains the nonnegative orthant, the cone of coefficients for SAGE polynomials admits the representation
\begin{align}
    \cpolysage{\emat} = \{ \vct{c} ~:& \text{ there exists }\vct{\hat{c}} \text{ in } \csage{\emat} \text{ where } \vct{\hat{c}} \leq \vct{c} \nonumber \\
    &~	 \text{and } \hat{c}_i \leq -c_i \text{ for all } i \text{ with } \evec_i \text{ not in } (2\mathbb{N})^n \}.\label{eq:c_poly_sage_def_3_efficient}
\end{align}
We use this representation to obtain the following theorem.

\begin{theorem}\label{thm:poly_sage_tractable}
    Let $L : \R^\ell \to \R^m$ be an injective affine map, $\emat \in \N^{n \times m}$ be a matrix of exponents ($n \leq m$), and $\vct{h}$ be a vector in $\R^\ell$. An $\epsilon$-approximate solution to
    \begin{equation}\label{eq:poly_runtime_formulation}
        \inf_{\vct{z} \in \R^\ell}\{ \vct{h}^\intercal \vct{z} \,:\, L(\vct{z}) \in \cpolysage{\emat} \}.
    \end{equation}
    can be computed in time $O(p(m)\log(1/\epsilon))$ for a polynomial $p$.
\end{theorem}
\begin{proof}
    We appeal to standard results on interior point methods (IPMs) for conic programming.
    The task is to show that $\cpolysage{\emat}$ can be expressed as a projection of a convex cone ``$K$,'' which possesses a tractable self-concordant barrier with a complexity parameter $\vartheta$ bounded by a polynomial in $m$.
    From there, the meaning of ``$\epsilon$-approximate'' and its relationship to the polynomial ``$p$'' depends highly on the details of a given IPM;
    relevant sources for general conic IPMs include \cite[\S 4]{nesterov1994} and \cite[\S 5]{Tunel2001}.
    In particular we rely on algorithms for optimizing over the \textit{exponential cone} $K_{\exp} = \cl\{ (u,v,w) \,:\, v \exp(u/v) \leq w, v > 0 \}$, and defer to \cite{Skajaa2014,expConeThesis,Papp2017} for formal meanings of ``$\epsilon$-approximate'' in our context.
    
    For each $i \in [m]$, let $\mtx{M}_i$ denote a matrix with ``$m_i$'' columns spanning $\ker(\emat_{\setminus i} - \evec_i\vct{1}^\intercal) \subset \R^{[m] \setminus i}$ and define the cone $K_{i} = \{  (\vct{u},\vct{v},t) \,:\, \vct{u},\vct{v} \in \R^{[m]\setminus i}_+,\, \relent{\vct{u}}{e \vct{v}} \leq t \}$.
    In terms of $\mtx{M}_i$ and $K_i$ we can reformulate the $i^{\text{th}}$ signomial AGE cone as
    \[
        \left\{ \vct{c}^{(i)} \,:\, \text{some } \vct{w}^{(i)} \in \R^{m_i} \text{ satisfies }
        \left(\mtx{M}_i\vct{w}^{(i)},~ \hat{\vct{c}}^{(i)}_{\setminus i},~ \hat{c}^{(i)}_i \right) \in K_i \right\}.
    \]
    Since $K_i$ can be represented with $m-1$ copies of $K_{\exp}$ and one linear inequality over $m-1$ additional scalar variables, the preceding display tells us that $\csage{\emat}$ can be represented with $m(m-1)$ copies of $K_{\exp}$, $m$ linear inequalities, and $O(m^2)$ scalar auxiliary variables.
    Combine this with the representation \eqref{eq:c_poly_sage_def_3_efficient} to find that the feasible set for \eqref{eq:poly_runtime_formulation} can be described with $O(m^2)$ exponential cone constraints, $O(m)$ linear inequalities, and $O(\ell + m^2) \in O(m^2)$ scalar variables.
    As the exponential cone has a tractable self-concordant barrier with complexity parameter $\vartheta_{\exp} = 3$, $\cpolysage{\emat}$ has a tractable self-concordant barrier with complexity parameter $O(m^2)$.
\end{proof}
} 

\subsection{Simple consequences of our signomial results}\label{subsec:polyCorollaries}

{\color{revisionBlue}Section \ref{subsec:sagepoly} suggested that the signomial SAGE cone is more fundamental than the polynomial SAGE cone.
This section serves to emphasize that idea, by showing how our study of the signomial SAGE cone quickly produces results in the polynomial setting.
The following corollaries are obtained by viewing Theorems \ref{thm:linear_indep_and_all_nonext_neg_is_exact} and \ref{thm:SageEqNng} through the lens of orthant-dominance.}
\begin{corollary}
    If exponent vectors $\emat$ induce a simplicial Newton polytope $\newpoly{\emat}$, and nonextremal exponents are linearly independent mod 2,
	then \textcolor{revisionBlue}{$\cpolysage{\emat} = \cnnp{\emat}$}.\label{corr:lin_indep_sage_poly}
\end{corollary}
\begin{corollary}
	Suppose $\emat$ belonging to $p = \Poly(\emat,\vct{c})$ can be partitioned into faces where (1) each simplicial face induces an orthant-dominated polynomial with at most two nonextremal {\color{revisionBlue}terms}, and (2) all other faces have at most one nonextremal {\color{revisionBlue}term}. Then $p$ is nonnegative iff it is SAGE. \label{cor:nearlySagePolyEqNNP}
\end{corollary}
{\color{revisionBlue}Unfortunately it is not possible to reduce the dependence of Corollary \ref{cor:nearlySagePolyEqNNP} on the coefficient vector $\vct{c}$ of the polynomial $p$.
The obstruction is that taking a signomial representative is not without loss of generality, as the case $\emat = [0, 1, 3, 4]$ shows.
}

{\color{revisionBlue}
To more deeply understand the polynomial SAGE cone it is necessary to study its extreme rays, as wells as its sparsity preservation properties.
We now show how this can be done by leveraging Theorems \ref{thm:restrictSS} and \ref{thm:extreme_rays} from Section \ref{sec:SAGEcontributions}.}
\begin{theorem}\label{thm:sage_poly_as_sum_of_age_poly}
	Defining the cone of ``AGE polynomials'' for exponents $\emat$ and index $k$ as
	\begin{align}
	\cpolyage{\emat,k} 
	\doteq \{   \vct{c} :&~\Poly(\emat,\vct{c}) \text{ is globally nonnegative, and} \nonumber \\
	&~ \vct{c}_{\setminus k} \geq \vct{0} , ~c_i = 0 \text{ for all } i \neq k \text{ with } \evec_i \not\in (2\mathbb{N})^n \},\label{eq:def_age_poly}
	\end{align}
	we have $\sum_{k=1}^m \cpolyage{\emat,k} = \cpolysage{\emat}$.
\end{theorem}
\begin{proof}
    {\color{revisionBlue}The inclusion $\cpolyage{\emat,k} \subset \cpolysage{\emat}$ is obvious, since polynomials satisfying \eqref{eq:def_age_poly} have AGE signomial representatives.
    We must show the inclusion $\cpolysage{\emat} \subset \sum_{k=1}^m \cpolyage{\emat,k}$.}

	Given a polynomial $p = \Poly(\emat,\vct{c})$, testing if $\vct{c}$ belongs to $\cpolysage{\emat}$ will reduce to testing if $\vct{\hat{c}}$ (given by Equation \eqref{eq:sig_rep_choose_c}) belongs to $\csage{\emat}$. 
	{\color{revisionBlue}Henceforth let $\vct{\hat{c}} \in \csage{\emat}$ be fixed and set $N = \{ i : \hat{c}_i < 0\}$.
	By Theorem \ref{thm:restrictSS}, there exist vectors $\{ \vct{\hat{c}}^{(i)} \in \cage{\emat,i} \}_{i \in N}$ where $\hat{c}_i^{(i)} = \hat{c}_i < 0$ for each $i$ and $\hat{c}_j^{(i)} = 0$ for all $j \in N \setminus \{i\}$.}
	The sign patterns here are important: $\vct{\hat{c}}^{(i)}$ is supported on the index set $ \{i \} \cup ( [m] \setminus N)$, and $\hat{c}_j^{(i)} \geq 0$ for all $j$ in $[m] \setminus N$.
	By construction of $\vct{\hat{c}}$, any index $j$ in $[m] \setminus N$ corresponds to an exponent vector $\evec_j$ in $(2\mathbb{N})^n$. 
	Therefore the carefully chosen vectors $\{ \vct{\hat{c}}^{(i)} \}_{i \in N}$ define not only AGE signomials, but also \textit{AGE polynomials} $\hat{p}_i = \Poly(\emat,\vct{\hat{c}}^{(i)} )$.
	{\color{revisionBlue}Lastly, for each index $i \in N$ set $\vct{c}^{(i)}$ by $\vct{c}_{\setminus i}^{(i)} = \vct{\hat{c}}_{\setminus i}^{(i)}$, and $c^{(i)}_i = -1 \cdot \text{sign}( c_i ) \cdot \hat{c}_i^{(i)}$. The resulting polynomials $p_i = \Poly(\emat,\vct{c}^{(i)})$ inherit the AGE property from $\hat{p}_i$, and sum to $p$.
	As we have decomposed our SAGE polynomial into an appropriate sum of ``AGE polynomials,'' the proof is complete.}
\end{proof}

\begin{corollary}\label{cor:poly_sparse_decomp}
    {\color{revisionBlue}Any SAGE polynomial can be decomposed into a sum of AGE polynomials in a manner that is cancellation-free.}
\end{corollary}
\begin{proof}
    {\color{revisionBlue}The cancelation-free decomposition is given constructively in the proof of Theorem \ref{thm:sage_poly_as_sum_of_age_poly}.}
\end{proof}

\begin{corollary}
	{\color{revisionBlue}If $\vct{c} \in \R^m$ generates an extreme ray of $\cpolysage{\emat}$, then $\{ \evec_i \,:\, i \in [m], c_i \neq 0 \}$ is either a singleton or a simplicial circuit.}\label{cor:extreme_rays_polynomial}
\end{corollary}
\begin{proof}
	{\color{revisionBlue}In view of Theorem \ref{thm:sage_poly_as_sum_of_age_poly}, it suffices to show that for fixed $k$ the extreme rays of $\cpolyage{\emat,k}$ are supported on single coordinates, or simplicial circuits.
	This follows from Theorem \ref{thm:extreme_rays}, since vectors in $\cpolyage{\emat,k}$ are -- up to a sign change on their $k^{\text{th}}$ component -- in 1-to-1 correspondence with vectors in $\cage{\mtx{\hat{\emat}},k}$, where $\mtx{\hat{\emat}}$ is obtained by dropping suitable columns from $\emat$.}
\end{proof}

\subsection{AM/GM proofs of nonnegativity, circuits, and SAGE}\label{subsec:amgm}

{\color{revisionBlue}
In 1989, Reznick defined an \textit{agiform} as any positive multiple of a homogeneous polynomial
$f = \Poly([\emat,\vct{\beta}], [\vct{\lambda},-1]^\intercal)$, where $\emat \in (2\N)^{n \times m}$ and $\vct{\beta} = \emat \vct{\lambda}$ for a weighting vector $\vct{\lambda} \in \Delta_m$ \cite{Reznick1989}.
Agiforms have AGE signomial representatives, which follows by plugging $\vct{\nu} = \vct{\lambda}$ into \eqref{eq:ageRelEnt}.
Reznick's investigation concerned extremality in the cone of nonnegative polynomials, and identified a specific subset of simplicial agiforms which met the extremality criterion  \cite[Theorem 7.1]{Reznick1989}.

Agiform-like functions were later studied by Paneta, Koeppl, and Craciun for analysis of biochemical reaction networks \cite[Proposition 3]{Pantea2012}.
Paneta et al. spoke in terms of \textit{posynomials} $f(\vct{x}) = \sum_{i=1}^{k+1} c_i \vct{x}^{\evec_i}$ where all $c_i \geq 0$;
a posynomial $f$ was said to \textit{dominate} the monomial $\vct{x}^{\vct{\beta}}$ if $\vct{x} \mapsto f(\vct{x}) - \vct{x}^{\vct{\beta}}$ was nonnegative on $\R^n_+$.
If we adopt the notation where $\Theta(\vct{c},\vct{\lambda}) = \prod_{i=1}^{k+1}(c_i / \lambda_i)^{\lambda_i}$, \cite[Theorem 3.6]{Pantea2012} says that in the case of a simplicial Newton polytope (i.e. $k=n$), monomial domination is equivalent to $1 \leq \Theta(\vct{c},\vct{\lambda})$ where $\vct{\lambda}$ gives the barycentric coordinates for $\vct{\beta} \in \newpoly{\emat}$.

A few years following Paneta et al., Iliman and de Wolff suggested taking sums of nonnegative circuit polynomials, which are nonnegative polynomials $f = \Poly([\emat,\vct{\beta}], [\vct{c},b]^\intercal)$ where $\{\evec_i\}_{i=1}^{n+1} \cup \{\vct{\beta}\}$ form a simplicial circuit \cite{SONC1}.
Iliman and de Wolff's Theorem 1.1 states that if all $\evec_i$ are even, $f$ is a circuit polynomial, and $\vct{\beta} \in \newpoly{\emat}$ has barycentric coordinates $\vct{\lambda} \in \Delta_{n+1}$, then $f$ nonnegative if and only if
\begin{equation}
    \text{ either } \quad |b| \leq  \Theta(\vct{c},\vct{\lambda}) \text{ and } \vct{\beta} \not\in (2\mathbb{N})^n \quad \text{ or } \quad -b \leq \Theta(\vct{c},\vct{\lambda}) \text{ and } \vct{\beta} \in (2\mathbb{N})^n. \label{eq:circuit_constraint}
\end{equation}
It is clear that \cite[Theorem 1.1]{SONC1} extends \cite[Theorem 3.6]{Pantea2012}, to account for sign changes of $b \cdot \vct{x}^{\vct{\beta}}$ and to impose no scaling on $|b|$.

The approach of taking sums of nonnegative circuit polynomials is now broadly known as ``SONC.''
Prior formulations for the SONC cone work by enumerating every simplicial circuit which could possibly be of use in a SONC decomposition (see \cite[\S 5.2]{SONC3}, and subsequently \cite{wangAGEpolynomial,wangSoncSupports}).
The circuit enumeration approach is extremely inefficient, as Example \ref{ex:exponential_circuits} shows an $m$-term polynomial can contain as many as $2^{(m-1)/2}$ simplicial circuits.

\noindent\begin{minipage}[t]{0.475\textwidth}
\strut\vspace*{-\baselineskip}\newline
    \includegraphics[width=\textwidth]{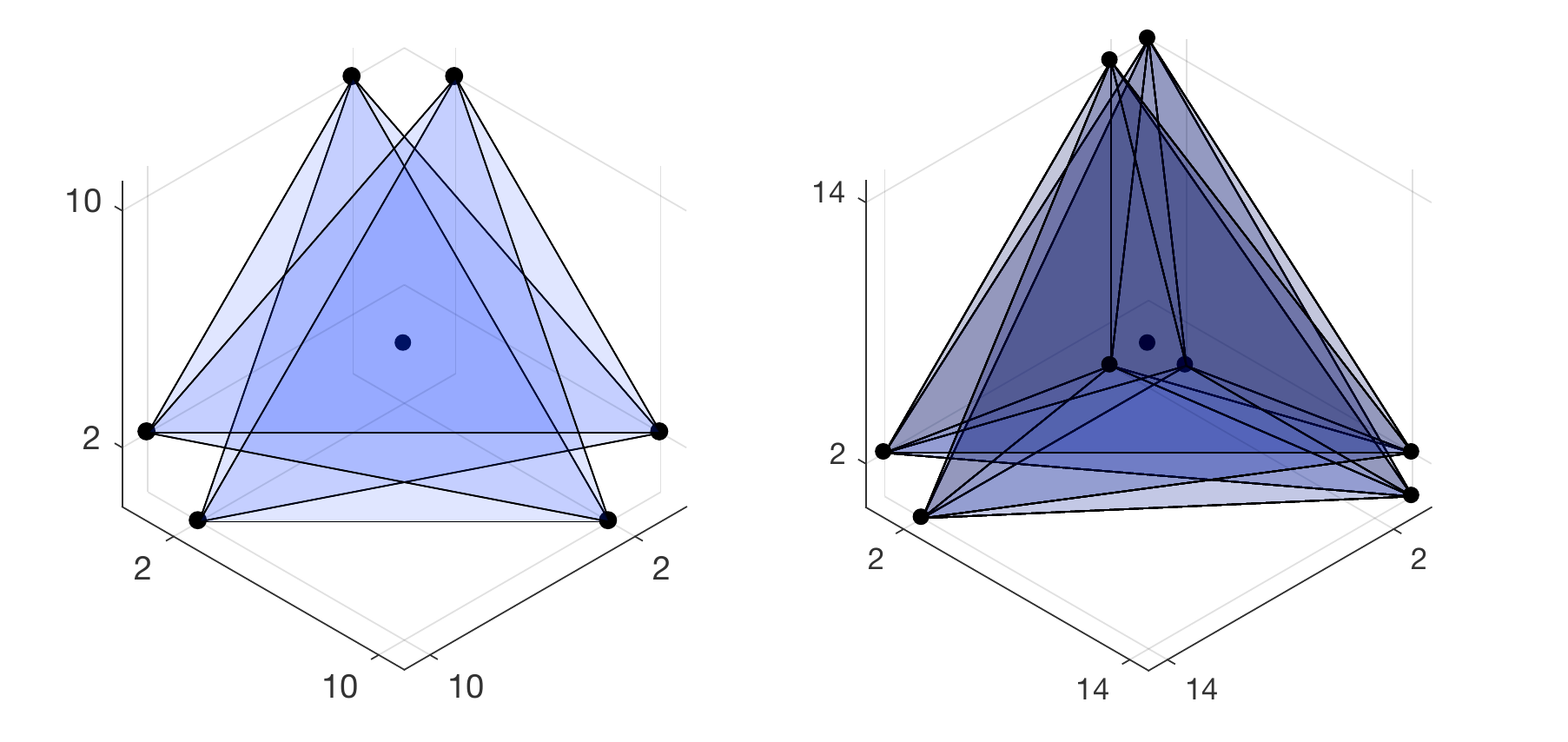}
\end{minipage}\hfill\begin{minipage}[t]{0.5\textwidth}
\strut\vspace*{-2\baselineskip}\newline
    \begin{example}\label{ex:exponential_circuits}
    Let $d$ be divisible by 2 and $n$. Construct an $n \times 2n$ matrix $\emat$ by setting $\evec_{2i-1}$ and $\evec_{2i}$ to distinct points in $\N^n \cap d \Delta_n$ adjacent to $d\vct{e}_i$. Then for large enough $d$, $\vct{\beta} = d\vct{1}/n$ will be contained in exactly $2^n$ simplices.
    Left: $(d,n)=(12,3)$, and a projection of $(d,n)=(16,4)$.
    \end{example}
\end{minipage}
Circuit enumeration is not merely a theoretical issue.
When using the heuristic circuit-selection technique from \cite{SONCExp}, Seidler and de Wolff's POEM software package fails to certify nonnegativity of the AGE polynomial $f(x,y) = (x-y)^2 + x^2 y^2$ and moreover only returns a bound $f^\star \geq -1$ \cite{poem:software}.

Of course- Corollary \ref{cor:extreme_rays_polynomial} tells us that a polynomial admits a SAGE certificate if and only if it admits a SONC certificate.
This is good news, since Theorem \ref{thm:poly_sage_tractable} says we can optimize over this set in time depending polynomially on $m$. In particular, we may avoid SONC's severe problems of circuit enumeration and circuit selection.
The qualitative distinction here is that while Paneta et al. and Iliman and de Wolff consider the weights $\vct{\lambda}$ as fixed (given by barycentric coordinates), the analogous quantity $\vct{\nu}$ in the SAGE approach is \textit{an optimization variable}.
At a technical level, the relative entropy formulation \eqref{eq:ageRelEnt} affords a joint convexity whereby SAGE can search simultaneously over coefficients $\vct{c}^{(i)}$ and weighting vectors $\vct{\nu}^{(i)}$.
As our proof of Theorem \ref{thm:poly_sage_tractable} points out, we can be certain that $[\emat_{\setminus i} - \evec_i \vct{1}^{\intercal}]\vct{\nu}^{(i)} = \vct{0}$ holds in exact arithmetic simply by defining $\vct{\nu}^{(i)} \leftarrow \mtx{M}_i \vct{w}^{(i)}$ for the indicated matrix $\mtx{M}_i$.

} 

\subsection{Comparison to existing results in the SONC literature} \label{subsec:sonc}
{\color{revisionBlue}
Due to the equivalence of the class of nonnegative polynomials induced by the SAGE and the SONC approaches, some of our results have parallels in the SONC literature. 

Corollary \ref{corr:lin_indep_sage_poly} is not stated in the literature, though it may be deduced from \cite[Corollary 7.5]{SONC1}.
Iliman and de Wolff prove \cite[Corollary 7.5]{SONC1} by signomializing $g(\vct{x}) = f(\exp \vct{x})$ and introducing an additional regularity condition so that $\nabla g(\vct{x}) = \vct{0}$ at exactly one $\vct{x} \in \R^n$.
Our proof of Corollary \ref{corr:lin_indep_sage_poly} stems from Theorem \ref{thm:linear_indep_and_all_nonext_neg_is_exact}, which employs a convex duality argument applicable to constrained signomial optimization problems in the manner of Corollary \ref{cor:constrained_case_exact}.

Wang showed that nonnegative polynomials in which at most one term $c_i \vct{x}^{\evec_i}$ takes on a negative value at some $\vct{x} \in \R^n$ (either $c_i < 0$ or $\evec_i \not\in (2\mathbb{N})^n$) are SONC polynomials \cite[Theorem 3.9]{wangAGEpolynomial}.
This result can be combined with the definition of AGE polynomial given in Theorem \ref{thm:sage_poly_as_sum_of_age_poly} in order to prove a weaker form of Corollary \ref{cor:extreme_rays_polynomial}, where all $\evec_i$ belong to $\ext\newpoly{\emat}$ or $\sint\newpoly{\emat}$.
We emphasize that Corollary \ref{cor:extreme_rays_polynomial} is not responsible for the major efficiency gains of SAGE from Theorem \ref{thm:poly_sage_tractable}; the SONC formulation in \cite[\S 5]{wangAGEpolynomial} uses $2^{(m-1)/2}$ circuits for the $m$-term polynomials from Example \ref{ex:exponential_circuits}.

Finally, in a result that was announced contemporaneously to the original submission of the present paper, Wang showed that summands in a SONC decomposition of a polynomial $f = \Poly(\emat,\vct{c})$ may have supports restricted to $\emat$ without loss of generality
\cite[Theorem 4.2]{wangSoncSupports}.
In light of the equivalence between the class of SAGE polynomials and of SONC polynomials, this result may be viewed as a weaker analog of our Corollary \ref{cor:poly_sparse_decomp}; specifically, \cite[Theorem 4.2]{wangSoncSupports} shows SONC certificates are sparsity-preserving but it does not provide a cancellation-free decomposition.

The distinctions between our polynomial Corollaries \ref{corr:lin_indep_sage_poly} and \ref{cor:nearlySagePolyEqNNP} versus our signomial Theorems \ref{thm:linear_indep_and_all_nonext_neg_is_exact} and \ref{thm:SageEqNng} make clear that polynomial results should not be conflated with signomial results.
With the exception of Section \ref{sec:PolyNNG}, our setup and results in this paper pertain to the class of signomials, which in general can have $\evec_i \in (\R \setminus \Q)^n$.
The developments in the SONC literature only consider polynomials, and employ analysis techniques of an algebraic nature which rely on integrality of exponents in fundamental ways (c.f. \cite[Theorem 4.2]{wangSoncSupports}).
In contrast, our techniques are rooted in convex duality and are applicable to the broader question of certifying signomial nonnegativity.
}  

\subsection{SAGE and SOS}\label{subsec:sos}

{\color{revisionBlue}
The SOS approach to polynomial nonnegativity considers polynomials $f$ in $n$ variables of degree $2d$, and attempts to express $f(\vct{x}) = L(\vct{x})^\intercal \mtx{P} L(\vct{x})$ where $\mtx{P}$ is a PSD matrix and $L : \R^n \to \R^{n + d \choose d}$ is a lifting which maps $\vct{x}$ to all monomials of degree at-most $d$ evaluated at $\vct{x}$ \cite{parilloPhD,Lasserre2001,shor}.
The identity $f(\vct{x}) = L(\vct{x})^\intercal \mtx{P} L(\vct{x})$ can be enforced with linear equations on the coefficients of $f$ and the entries of $\mtx{P}$, so deciding SOS-representability reduces to a semidefinite program.

Because it is extremely challenging to solve semidefinite programs at scale, several modifications to SOS have been proposed to offer reduced complexity.
Kojima et. al built on earlier work of Reznick \cite{extremalPSDFormsWithFewTerms} to replace the lifting ``$L$'' appearing in the original SOS formulation with a smaller map using fewer monomials \cite{sparsesosDecideSupports2004}. 
Their techniques had meaningful use-cases, but could fail to perform any reduction in some very simple situations \cite[Proposition 5.1]{sparsesosDecideSupports2004}.
Subsequently, Waki et. al introduced the \textit{correlative sparsity} heuristic to induce structured sparsity in the matrix variable $\mtx{P}$ \cite{sparsesosChordal2006}.
Shortly thereafter Nie and Demmel suggested replacing the standard lifting by a collection of smaller $\{L_i\}_i$, so as to express $f(\vct{x}) = \sum_i L_i(\vct{x})^\intercal \mtx{P}_i L_i(\vct{x})$ with order ${k + d \choose d}$ PSD matrices $\mtx{P}_i$ for some $k \ll n$ \cite{sparsesosDemmel2008}.
Very recently, Ahmadi and Majumdar suggested one use the standard lifting together with a \textit{scaled diagonally dominant} matrix $\mtx{P}$ of order ${n+d \choose d}$;\footnote{\textcolor{revisionBlue}{A symmetric matrix $\mtx{P}$ is scaled-diagonally-dominant if there exists diagonal $\mtx{D} \succ \mtx{0}$ so that $\mtx{D} \mtx{P} \mtx{D}^\intercal$ is diagonally dominant.
Such matrices can be represented as a sum of $2 \times 2$ PSD matrices with appropriate zero padding.}} these ``SDSOS polynomials'' are precisely those polynomials admitting a decomposition as a sum of binomial squares \cite{SDSOS1}.

Each of these SOS-derived works suffers from a drawback that SOS decompositions may require cancellation on coefficients of summands $f_i = g_i^2$ as one recovers $f= \sum_{i} f_i$.
As a concrete example, consider $f(x,y) = 1 - 2 x^2 y^2 + {\color{revisionBlue}x^8/2 + y^8/2}$; this polynomial is nonnegative (in fact, AGE) and admits a decomposition as a sum of binomial squares.
The trouble is that to decompose $f$ as a sum of binomial squares, the summands $f_i = g_i^2$ require additional terms $+x^4 y^4$ and $-x^4 y^4$.
By contrast, SAGE certificates need only involve the original monomials in $f$, and one may take summand AGE polynomials to be cancellation-free with no loss of generality (Corollary \ref{cor:poly_sparse_decomp}).
The SAGE approach also has the benefit of being formulated with a relative entropy program of size $O(m^2)$ (Theorem \ref{thm:poly_sage_tractable}),
while SOS-derived works have complexity scaling exponentially with a polynomial's degree $d$.

We make two remarks in closing.
First, it is easy to verify that every binomial square is an AGE polynomial, and so SAGE can certify nonnegativity of all SDSOS polynomials.
Second, it is well known that proof systems leveraging the AM/GM inequality (SAGE among them) can certify nonnegativity of some polynomials which are \textit{not} SOS.
A prominent example here is the Motzkin form $f(x,y,z) = x^2 y^4 + x^4 y^2 + z^6 - 3 x^2 y^2 z^2$.

}  

\subsection{Extending SAGE polynomials to a hierarchy}\label{subsec:polyhier}

{\color{revisionBlue}We conclude this section by discussing how to obtain hierarchies for constrained polynomial optimization problems, in a manner which is degree-independent and sparsity preserving.
Adopt the standard form \eqref{eq:constr_opt_standard_form} for minimizing a polynomial $f$ subject to constraint polynomials $\{g_i\}_{i=1}^k$.}
Here, all polynomials are over a common set of exponents \textcolor{revisionBlue}{$\emat \in \N^{n \times m}$}, with $\evec_1 = \vct{0}$ \textcolor{revisionBlue}{and $n \leq m$}.
Our development is based on a hierarchy for signomials that is described in \cite[\S 3.3]{CS16}.

Consider operators $\mathcal{A}$ and $\mathcal{C}$ taking values $\mathcal{A}\left( \Poly(\emat,\vct{c}) \right) = \emat$ and $\mathcal{C}\left( \Poly(\emat,\vct{c}) \right) = \vct{c}$ respectively. 
We shall say {\color{revisionBlue}our} SAGE polynomial hierarchy is indexed by two parameters: $p$ and $q$.
The parameter $p$ controls the complexity of Lagrange multipliers; when $p = 0$, the Lagrange multipliers are simply $\lambda_i \geq 0$.
For general $p$, the Lagrange multipliers are SAGE polynomials over exponents $\emat' \doteq \mathcal{A}\left( \Poly(\emat,\vct{1})^p \right)$.
The parameter $q$ controls the number of constraints in the nonconvex primal problem: {\color{revisionBlue}$H = \{ h_i \}_{i=1}^{k^q}$} are obtained by taking all $q$-fold products of the $g_i$.
Once the Lagrangian {\color{revisionBlue}$\mathcal{L} = f - \gamma - \sum_{h \in H}h \cdot s_h$} is formed, it will be a polynomial over exponents $\emat'' \doteq \mathcal{A}\left( \Poly(\emat,\vct{1})^{p + q} \right)$.
{\color{revisionBlue}By the minimax inequality we have}
\[
{\color{revisionBlue}(f,g)^{(p,q)} \doteq \sup_{\gamma,\{s_h\}_{h \in H}}\{ \gamma \,:\, \mathcal{C}(\mathcal{L}) \in \cpolysage{\emat''},~\mathcal{C}(s_h) \in \cpolysage{\emat'}, \forall \, h\, \in H\} \leq (f,g)^\star.}
\]
\textcolor{revisionBlue}{Following Theorem \ref{thm:poly_sage_tractable}, the above can be solved in time polynomial in $m,k$ for each fixed $p,q$.}
\textcolor{revisionBlue}{As} $p$ and $q$ increase, we obtain improved bounds at the expense of an increase in computation.
Mirroring \cite{CS16}, one can appeal to representation theorems from the real algebraic geometry literature \cite{Schweighofer2002,hierconv1,hierconv2} to prove that {\color{revisionBlue}this hierarchy} can provide arbitrarily accurate lower bounds \textcolor{revisionBlue}{for} sparse polynomial optimization problems in which the constraint set is {\color{revisionBlue}Archimedian (for example, if all variables have explicit finite upper and lower bounds).

Our broader message here -- beyond results on convergence to the optimal value of specific hierarchies -- is that the above construction qualitatively differs from other hierarchies in the literature, because the optimization problems encountered at every level of our construction \emph{depend only on the nonnegative lattice generated by the original exponent vectors $\emat$}.
The theoretical underpinnings of this sparsity-preserving hierarchy trace back to the decomposition result given by Theorem \ref{thm:restrictSS}.
Thus, it is possible to obtain entire families of relative entropy relaxations that are sparsity-preserving, which reinforces our message about the utility of SAGE-based relative entropy optimization for sparse polynomial problems.
}

\section{Towards Necessary and Sufficient Conditions for SAGE versus Nonnegativity}\label{sec:constructCounterExAndDualChar}

We conclude this paper with a discussion on the extent to which our results tightly characterize the distinction between SAGE and nonnegativity for signomials.  This section is split into three parts.
In the first part, we describe a process for identifying cases where $\csage{\emat} \subsetneq \cnns{\emat}$.
This process is illustrated with several examples which suggest that our results from Section \ref{sec:structureSAGEandNNG} are essentially tight.
Section \ref{subsec:conjecturecantimprove} presents a formal conjecture regarding the ways in which our results might be improved, {\color{revisionBlue}and} Section \ref{subsec:dualChar} {\color{revisionBlue}provides} a novel dual formulation for when $\csage{\emat} = \cnns{\emat}$.

\subsection{Constructing examples of non-equality}\label{subsec:constructEx}

Given a matrix of exponent vectors $\emat$, we are interested in finding a coefficient vector $\vct{c}$ so that $f = \Sig(\emat, \vct{c})$ satisfies $f_{\mathsf{SAGE}} < f^\star$. 
If such $\vct{c}$ exists, then it is evident that $\csage{\emat} \neq \cnns{\emat}$.

The na\"ive approach to this process would be to carefully construct signomials where the infimum $f^\star$ is known by inspection, to compute $f_{\mathsf{SAGE}}$, and then to test if the measured value $|f_{\mathsf{SAGE}} - f^\star|$ is larger than would be possible from rounding errors alone. 
A serious drawback of this approach is that it can be quite difficult to construct $\emat$ and $\vct{c}$ where  $f^\star$ is apparent, and yet $\{ \evec_i : c_i \neq 0 \}$ satisfy the properties for the conjecture under test.

To address this challenge, we appeal to the idea alluded to in Section \ref{sec:SAGE} that SAGE provides a means of computing {\color{revisionBlue}a} \textit{sequence} of lower bounds $(f_{\mathsf{SAGE}}^{(\ell)})_{\ell \in \mathbb{N}}$.
For details on this {``\color{revisionBlue} unconstrained} SAGE hierarchy,'' we refer the reader to \cite{CS16}. For our purposes, suffice it to say that
\[
f_{\mathsf{SAGE}}^{(\ell)} \doteq \sup\{ \gamma : \Sig(\emat, \vct{1})^\ell (f- \gamma) \text{ is SAGE } \}
\]
defines a non-decreasing sequence bounded above by $f^\star$.
Thus, while we cannot readily check if $| f_{\mathsf{SAGE}} - f^\star | \gg 0$, we \textit{can} compute a few values of $f_{\mathsf{SAGE}}^{(\ell)}$ for $\ell > 0$, and check if $|f_{\mathsf{SAGE}}^{(0)} - f_{\mathsf{SAGE}}^{(\ell)}| \gg 0$.

The remainder of this section goes through case studies in which we probe the sensitivity our earlier theorems' conclusions to their stated assumptions. 
All computation was performed {\color{revisionBlue}with a late 2013 MacBook Pro with a 2.4GHz i5 processor,} using \texttt{CVXPY} \cite{cvxpy,cvxpy_rewriting} as an interface to the conic solver \texttt{ECOS} \cite{expConeThesis,ecos}.\footnote{\textcolor{revisionBlue}{Code is hosted at \texttt{github.com/rileyjmurray/sigpy}, and also \texttt{data.caltech.edu/records/1427}.}}
Numerical precision is reported to the farthest decimal point where the primal and dual methods for computing $\sagerelax{f}{\ell}$ agree.

\begin{example}\label{ex:counter_ex1}
    We test here whether it is possible to relax the assumption of simplicial Newton polytope in Theorem \ref{thm:linear_indep_and_all_nonext_neg_is_exact}.
    Since every Newton polytope in $\mathbb{R}$ is trivially simplicial, the simplest signomials available to us are over $\mathbb{R}^2$. 
    With that in mind, consider
    \[
    \emat = 	\begin{bmatrix} 
    0 ~& 2 ~& 1  ~& 0 ~& 0 ~& 2\\
    0 ~& 0 ~& 0  ~& 2 ~& 1 ~& 2
    \end{bmatrix}.
    \]
    This choice of $\emat$ is particularly nice, because were it not for the last column $\evec_6 = [2,~2]^\intercal$, we would very clearly have $\csage{\emat} = \cnns{\emat}$.
    {\color{revisionBlue}We test tested a few values for $\vct{c}$ before finding}
    \[
    \vct{c} = [0, ~3, -4, ~2, -2, ~1]^\intercal,
    \]
    which resulted in $f_{\mathsf{SAGE}}^{(\ell)} \approx -1.83333$, and $f_{\mathsf{SAGE}}^{(\ell)} \approx -1.746505595 = f^\star$.
    Because the absolute deviation $|f_{\mathsf{SAGE}} - f^\star| \approx 0.08682$ is much larger than the precision to which we solved these relaxations, we conclude that $\csage{\emat} \neq \cnns{\emat}$ for this choice of $\emat$.
\end{example}

\begin{example}\label{ex:counter_ex2}
    Let us reinforce the conclusion from Example \ref{ex:counter_ex1}.
    Applying a 180 degree rotation about the point (1,1) to the columns of $\emat$, we obtain
    \[
    \emat = \begin{bmatrix} 
    0 ~& 2 ~& 0 ~& 2 ~& 1 ~& 2 \\
    0 ~& 0 ~& 2 ~& 2 ~& 2 ~& 1   
    \end{bmatrix}.
    \]
    We then choose the {\color{revisionBlue}coefficients} in a manner informed by the theory developed in Section \ref{subsec:dualChar}
    \[
    \vct{c} = [0, ~1, ~1, ~1.9, -2, -2]^\intercal
    \]
    which subsequently defines $f = \Sig(\emat, \vct{c})$.
    In this case the primal formulation for $f_{\mathsf{SAGE}}$ is infeasible, and so $\sagerelax{f}{0} = -\infty$.
    Meanwhile, the second level of the {\color{revisionBlue}unconstrained} hierarchy produces $\sagerelax{f}{1} \approx -0.122211863 = f^\star$. 
    Thus in a very literal sense, the gap $| f_{\mathsf{SAGE}} - f^\star |$ could not be larger.
\end{example}

We know from Theorem \ref{thm:SageEqNng} that any signomial with at most four terms is nonnegative if and only if it is SAGE.
It is natural to wonder if in some very restricted setting (e.g. univariate signomials) the SAGE and nonnegativity cones would coincide for signomials with five or more terms{\color{revisionBlue}; Example \ref{ex:counter_ex3} shows this is not true in general.}

\begin{example}\label{ex:counter_ex3}
    {\color{revisionBlue}For $f = \Sig(\emat,\vct{c})$ with $\emat = [0, ~1, ~2, ~3, ~4]$ and $\vct{c} = [1,-4,~7,-4,~1]^\intercal$, we have
    $\sagerelax{f}{0} \approx -0.3333333$ and $\sagerelax{f}{1} \approx 0.2857720944$.
    Per the affine-invariance invariance properties of the SAGE and nonnegativity cones, this examples shows $\csage{\emat}$ is a strict subset of $\cnns{\emat}$ for every $1 \times 5$ matrix $\emat$ with equispaced values.}
\end{example}

{\color{revisionBlue}Together, Examples \ref{ex:counter_ex1} through \ref{ex:counter_ex3} demonstrate there are meaningful senses in which Theorems \ref{thm:linear_indep_and_all_nonext_neg_is_exact} through \ref{thm:SageEqNng} cannot be improved upon.}

\subsection{A conjecture, under mild regularity conditions}\label{subsec:conjecturecantimprove}

Despite the conclusion in the previous subsection, there \textit{are} settings when we can prove $\csage{\emat} = \cnns{\emat}$ in spite of $\emat$ not satisfying the assumptions of Theorem \ref{thm:SageEqNng}.  For example, one case in which SAGE equals nonnegativity is when $\emat = [\vct{0}, \mtx{I}, \mtx{D}]$ where $\mtx{D}$ is a diagonal matrix with diagonal entries in $(0,1)$.  Here one proves equality as follows: for each possible sign pattern of $\vct{c} \in\cnns{\emat}$, there exists a lower dimensional simplicial face $F$ of $\newpoly{\emat}$ upon which we invoke Theorem \ref{thm:linear_indep_and_all_nonext_neg_is_exact}, 
and for which the remaining exponents (those outside of $F$) have positive coefficients. 
We know that the signomial induced by the exponents outside of $F$ is trivially SAGE, 
and so by Theorem \ref{thm:needNonnegOfFaces} we conclude $\vct{c} \in\csage{\emat}$. 
As this holds for all possible sign patterns on $\vct{c} $ in $\cnns{\emat}$, we have $\csage{\emat} = \cnns{\emat}$. However, this case is somewhat degenerate, and we wish to exclude it in our discussion via some form of regularity on $\emat$.

The most natural regularity condition on $\emat$ would be that it admits only the trivial partition, and indeed we focus on the case when every $\evec_i$ belongs to either $\ext \newpoly{\emat}$ or $\sint \newpoly{\emat}$. 
In this setting we have the following corollary of Theorem \ref{thm:SageEqNng}.

\begin{corollary}\label{cor:FeqPSageEqNng} 
	If $\newpoly{\emat}$ is full dimensional with either
	\begin{enumerate}
		\item at most one interior exponent, or
		\item $n+1$ extreme points and at most two interior exponents
	\end{enumerate}
	then $\csage{\emat} = \cnns{\emat}$. 
\end{corollary}
Along with this corollary, we present a conjecture for the reader's consideration.
\begin{conjecture}\label{conj:FeqPSage_not_EqNng}  
	If $\newpoly{\emat}$ has every $\evec_i$ in either $\ext\newpoly{\emat}$ or $\sint\newpoly{\emat}$, but $\emat$ does not satisfy the hypothesis of Corollary \ref{cor:FeqPSageEqNng}, then $\csage{\emat} \neq \cnns{\emat}$.
\end{conjecture}
Note that when $\emat$ satisfies the stated assumptions and and further has some $\evec_i = \vct{0}$ in the interior, Theorem \ref{thm:fsageisbounded1pluseps} ensures that $f = \Sig(\emat, \vct{c})$ can have $f_{\mathsf{SAGE}}$ deviate from $f^\star$ only by a finite amount.
To overcome a potential obstacle posed by this result in the resolution of {\color{revisionBlue}Conjecture \ref{conj:FeqPSage_not_EqNng}}, one can also consider modifying the hypotheses of the conjecture to require that all $\evec_i$ lie in the relative interior of the Newton polytope.

To finish discussion on Conjecture \ref{conj:FeqPSage_not_EqNng}, we provide empirical support with the following examples.

\begin{example}
    Let $f$ be a signomial in two variables with 
    \begin{equation*}
    \begin{bmatrix} \emat \\ \hline \vct{c}^\intercal \end{bmatrix}  = \begin{bmatrix} 
    0 & 1 & 0 & 0.30 & 0.21 & 0.16 \\
    0 & 0 & 1 & 0.58 & 0.08 & 0.54 \\ \hline
    33.94 & 67.29 & 1 & 38.28 & -57.75 & -40.37  
    \end{bmatrix}.
    \end{equation*}
    Then $f_{\mathsf{SAGE}} = -24.054866 < \sagerelax{f}{1} = -21.31651$. 
    This example provides the minimum number of interior exponents needed to be relevant to Conjecture \ref{conj:FeqPSage_not_EqNng} in the simplicial case.
\end{example}

\begin{example}
     Let $f$ be a signomial in two variables with
    \begin{equation*}
    \begin{bmatrix} \emat \\ \hline \vct{c}^\intercal \end{bmatrix}  = \begin{bmatrix} 
    0 & 1 & 0 & 2 & 0.52 & 1.30 \\
    0 & 0 & 1 & 2 & 0.15 & 1.38 \\ \hline
    0.31 & 0.85 & 2.55 & 0.65 & -1.48 & -1.73  
    \end{bmatrix}.
    \end{equation*}
    then $f_{\mathsf{SAGE}} = 0.00354263 < \sagerelax{f}{1} = 0.13793126$.
    This signomial has the minimum number of interior exponents needed to be relevant to Conjecture \ref{conj:FeqPSage_not_EqNng} in the nonsimplicial case.
\end{example}

\subsection{A dual characterization of SAGE versus nonnegativity}\label{subsec:dualChar}

In this section, we provide a general necessary and sufficient dual characterization in terms of certain moment-type mappings for the question of $\csage{\emat} = \cnns{\emat}$.
{\color{revisionBlue} To establish this dual characterization we use some new notation.
Given two vectors $\vct{u}$, $\vct{v}$ the Hadamard product $\vct{w} = \vct{u} \circ \vct{v}$ has entries $w_i  = u_i v_i$; this is extended to allow sets in either argument in the same manner as the Minkowski sum.
The operator $\mathcal{R}(\cdot)$ returns the range of a matrix.}

We begin with the following proposition (proven in the appendix).
\begin{proposition}\label{prop:tfae}
	If $\emat$ has $\evec_1 = \vct{0}$, then the following are equivalent:
	\begin{itemize}
		\item[1.] For every vector $\vct{c}$, the function $f = \Sig(\emat,\vct{c})$ satisfies $f^\star = f_{\mathsf{SAGE}}$.
		\item[2.] $ \cnns{\emat} = \csage{\emat} $.
		\item[3.] $\{ \vct{v} : v_1 = 1, ~ \vct{v} \mathrm{~in~} \csage{\emat}^\dagger  \} \subset \cl\conv \exp \mathcal{R}(\emat^\intercal) $.
	\end{itemize}
\end{proposition}

Our dual characterization consists of two new sets, both parameterized by $\emat$.
The first of these sets relates naturally to the third condition in Proposition \ref{prop:tfae}. 
Formally, the \textit{moment preimage} of some exponent vectors $\emat$ is the set 
\[
T(\emat) \doteq \log\cl\conv\exp \mathcal{R}(\emat^\intercal).
\]
Here, we extend the logarithm to include $\log 0 = -\infty$ in the natural way.
The second set appearing in our dual characterization is defined less explicitly.
For a given $\emat$, we say that $S(\emat)$ is a set of \textit{SAGE-feasible slacks} if $f = \Sig(\emat, \vct{c})$ has
\[
f_{\mathsf{SAGE}} = \inf\{ \vct{c}^\intercal \exp \vct{y} : \vct{y} \text{ in } \mathcal{R}(\emat^\intercal) + S(\emat) \}
\]
for every $\vct{c}$ in $\mathbb{R}^m$.

\begin{theorem}
	Let $\emat$ have $\evec_1 = \vct{0}$, and let $S(\emat)$ be any set of SAGE-feasible slacks over exponents $\emat$. Then $\csage{\emat} = \cnns{\emat}$ iff $S(\emat) \subset T(\emat)$.  \label{thm:S_alpha_subset_of_THING_is_exact}
\end{theorem}
\begin{proof}[Proof of Theorem \ref{thm:S_alpha_subset_of_THING_is_exact}]
	To keep notation compact write $U = \mathcal{R}(\emat^\intercal)$ and $S = S(\emat)$.
	Also, introduce $W = \{\vct{v} : v_1 = 1, \vct{v} \text{ in } \csage{\emat}^\dagger \}$ to describe the feasible set to the dual formulation for $f_{\mathsf{SAGE}}$.
	By the supporting-hyperplane characterizations of convex sets, the definitions of $S$ and $W$ ensure
	\[W = \cl\conv\exp(\mathcal{U} + S).\]
	Thus by the equivalence of \textit{1} and \textit{3} in Proposition \ref{prop:tfae}, it follows that all SAGE relaxations will be exact if and only if $\exp(U + S) \subset \cl\conv\exp U$.
	We apply a pointwise logarithm to write the latter condition as $U + S \subset \log\cl\conv\exp U$. 
	
	Now we prove that $T \doteq \log\cl\conv\exp U$ is invariant under translation by vectors in $U$.
	It suffices to show that $\exp(\vct{v} + T ) = \exp T$ for all vectors $\vct{v}$ in $U$.
	Fixing $\vct{v} $ in $U$ we have
	\begin{align*}
	\exp(\vct{v} + T) &= \exp(\vct{v}) \circ \exp(T) \\
	&= \exp(\vct{v}) \circ \cl\conv \exp (U) \\
	&= \cl\conv\exp(\vct{v} + U) \\
	&= \cl\conv\exp(U) = \exp(T)
	\end{align*} 
	as claimed.
	This translation invariance establishes that $U + S \subset \log\cl\conv\exp U$ is equivalent to $S \subset \log\cl\conv\exp U$, and in turn that condition \textit{1} of Proposition \ref{prop:tfae} holds if and only if $S \subset \log\cl\conv\exp U$.
	The claim now follows by the equivalence of \textit{1} and \textit{2} in Proposition \ref{prop:tfae}. 
\end{proof}

It is the authors' hope that Theorem \ref{thm:S_alpha_subset_of_THING_is_exact} may help future efforts to resolve Conjecture \ref{conj:FeqPSage_not_EqNng}.
A starting point in understanding the moment preimage could be to use cumulant generating functions from probability theory.
For constructing sets of SAGE-feasible slacks, one might use a change-of-variables argument similar to that seen in the proof of Theorem \ref{thm:linear_indep_and_all_nonext_neg_is_exact}.

\subsection*{Acknowledgments}
    {\color{revisionBlue}The authors are thankful for the detailed suggestions of anonymous referees, which have led to a much-improved revision of our original manuscript.}
	V.C. would like to acknowledge helpful conversations with Parikshit Shah, particularly on the connections between SAGE and SDSOS polynomials.  R.M. was supported in part by NSF grant CCF-1637598 and by an NSF Graduate Research Fellowship.  V.C. was supported in part by NSF grants CCF-1350590 and CCF-1637598, AFOSR grant FA9550-16-1-0210, and a Sloan Research Fellowship.  A.W. was supported in part by NSF grant CCF-1637598.
\bibliography{sources}  

\section{Appendix}

\subsection{Proof of Lemma \ref{lem:decompose_mixture}}

\decomposeMixture*

\begin{proof}
    {\color{revisionBlue}Denote $\supp_{\mtx{B}}(\vct{\lambda}) = \{ \vct{b}_j \,:\,j \in [d],\, \lambda_j \neq 0 \}$. The proof is constructive, where there is nothing to prove when $\supp_{\mtx{B}}\vct{\lambda}$ is simplicial.
    Suppose then that  $\vct{\lambda} \in \Lambda$ has \textit{nonsimplicial} $\supp_{\mtx{B}}(\vct{\lambda})$.
    We show that it is possible to decompose $\vct{\lambda} = z \vct{\lambda}^{(1)} + (1-z)\vct{\lambda}^{(2)}$ for some $z \in (0, 1)$ and $\vct{\lambda}^{(i)} \in \Lambda$ where $\supp \vct{\lambda}^{(i)} \subsetneq \supp \vct{\lambda}$.
    It should be clear that if this is possible, then the process may be continued in a recursive way if either $\supp_{\mtx{B}}(\vct{\lambda}^{(i)})$ are nonsimplicial, and so the claim would follow.
    
	The statement ``$\vct{\lambda} \in \Lambda$'' means that $\vct{h}$ may be expressed as a convex combination of vectors in $\supp_{\mtx{B}}(\vct{\lambda})$, and so by
	Minkowski-Carath\'{e}odory,} there exists at least one $\vct{\lambda}^{(1)} $ in $ \Lambda_x$ with $\supp \vct{\lambda}^{(1)} \subsetneq \supp \vct{\lambda} $ and simplicial $\supp_{\mtx{B}}(\vct{\lambda}^{(1)})$. We will use $\vct{\lambda}$ and $\vct{\lambda}^{(1)}$ to construct the desired $\vct{\lambda}^{(2)}$ and $z$.
	
	For each real $t$, consider $\vct{\lambda'}_t \doteq \vct{\lambda}^{(1)} + t(\vct{\lambda} - \vct{\lambda}^{(1)})$. 
	It is easy to see that for all $t$ the vector $\vct{\lambda'}_t $ belongs to the affine subspace $ \{\vct{w} : \vct{h} = \mtx{B}\vct{w},~ \mathbf{1}^\intercal \vct{w} = 1\}$,  and furthermore the support of  $\vct{\lambda'}_t $ is contained within the support of $\vct{\lambda}$.
	Now define $T = \max\{t : \vct{\lambda'}_t \text{ in } \Delta_d\}$; we claim that $T > 1$ and that the support of $ \vct{\lambda'}_T$ is a proper subset of the support of $\vct{\lambda}$.
	The latter claim is more or less immediate.
	To establish the former claim consider how $\vct{\lambda'}_t$ (as an affine combination of $\vct{\lambda}^{(1)},\vct{\lambda}$) belongs to $\Delta_d$ if and only if it is elementwise nonnegative.
	This lets us write $T = \max\{ t : \vct{\lambda'}_t \geq \vct{0}   \}$. 
	Next, use our knowledge about the support of $\vct{\lambda'}_t$ to rewrite the constraint ``$\vct{\lambda'}_t \geq \vct{0}$'' as ``${\lambda}_i^{(1)} + t({\lambda}_i - {\lambda}^{(1)}_i) \geq 0 \text{ for all } i \text{ in } \supp \vct{\lambda}$.''
	Once written in this form, we see that for $t = 1$ all constraints are satisfied strictly. 
	It follows that $T > 1$ at optimality, and furthermore that the support of $\vct{\lambda'}_T$ is distinct from (read: \textit{a proper subset of}) that of $\vct{\lambda}$.
	
	We complete the proof by setting $\vct{\lambda}^{(2)} = \vct{\lambda'}_T$ and $z = 1-1/T$. 
\end{proof}

\subsection{Proof of Lemma \ref{lem:ifgisface}}

\partitionLemma*

\begin{proof}[Proof of Lemma \ref{lem:ifgisface}]
    {\color{revisionBlue}Denote $f = \Sig(\emat,\vct{c})$ and $g = \Sig_F(\emat,\vct{c})$.}
	For brevity write $P = \newpoly{\emat}$; we may assume without loss of generality that $P$ contains the origin. 
	If $F = P$ then $g = f$ and the claim is trivial. 
	If otherwise, the affine hull of $F$ must have some positive codimension $\ell$, and there exist supporting hyperplanes $\{S_i\}_{i=1}^\ell$ such that $F = [ \cap_{i=1}^\ell S_i] \cap P$. 
	We can express $S_i$ as $\{\vct{x} :  \vct{s}_i^\intercal \vct{x} = r_i\}$ for a vector $\vct{s}_i $ and a scalar $r_i$. 
	Because $P$ is convex we know that it is contained in one of the half spaces $\{ \vct{x} : \vct{s}_i^\intercal \vct{x}  \leq r_i  \}$ or $\{ \vct{x} : \vct{s}_i^\intercal \vct{x}  \geq r_i \}$. 
	By possibly replacing $(\vct{s}_i,r_i)$ by $(-\vct{s}_i,-r_i)$, we can assume that $P$ is contained in $ \{ \vct{x} : \vct{s}_i^\intercal \vct{x}  \leq r_i  \}$. 
	In addition, the assumption that $\vct{0}$ belongs to $P$ ensures that each $r_i$ is nonnegative.
	Now define $\vct{s} = \sum_{i=1}^\ell \vct{s}_i$ and $r = \sum_{i=1}^\ell r_i \geq 0$. The pair $(\vct{s}, r)$ is constructed to satisfy the following properties:
	\begin{itemize}
		\item For every $\evec_j $ in $ F$, we have $\evec_j^\intercal \vct{s} = r$. 
		\item For every $\evec_j $ \textit{not} in $ F$, we have $\evec_j^\intercal \vct{s} < r$. 
	\end{itemize}
	Finally, define $h = \Sig_{P \setminus F}(\emat,\vct{c})$ so $f = g + h$.
	The remainder of the proof is case analysis on $r$.
	
	If $r = 0$ then we must have $r_i = 0$ for all $i$. 
	The condition that $r_i = 0$ for all $i$ implies that $F$ is contained in a linear subspace $U$ which is orthogonal to $\vct{s}$, and so nonnegativity of $g$ over $\mathbb{R}^n$ reduces to nonnegativity of $g$ over $U$. 
	Suppose then that there exists some $\vct{\hat{x}} $ in $U$ where $g(\vct{\hat{x}})$ is negative. 
	For any vector $\vct{y} $ in the orthogonal complement of $U$ we have $g(\vct{\hat{x}} + \vct{y}) = g(\vct{\hat{x}})$. Meanwhile no matter the value of $\vct{\hat{x}}$ we know that $\lim_{t \to \infty} h(\vct{\hat{x}} + t\vct{s}) = 0$. 
	Using $f^\star \leq \inf\{ f(\vct{\hat{x}} + t\vct{s}) : t \text{ in } \mathbb{R}  \} \leq g(\vct{\hat{x}})$, we have the desired result for $r = 0$: $g^\star < 0$ implies $f^\star < 0$.
	
	Now consider the case when $r$ is positive.
	Define the vector $\vct{\hat{s}} = r \vct{s} / \|\vct{s}\|^2$; we produce an upper bound on $f^\star$ by searching over all hyperplanes $\{\vct{x} :   \vct{\hat{s}}^\intercal \vct{x}  = t \}$ for $t $ in $ \mathbb{R}$. 
	Specifically, for any $\vct{x} $ in $ \mathbb{R}^n$ there exists a scalar $t$ and a vector $\vct{y}$ such that $\vct{x} = t\vct{\hat{s}} + \vct{y}$ and $\vct{\hat{s}}^\intercal \vct{y} = 0$. 
	In these terms we have 
	\begin{align}
	g(t \vct{\hat{s}} + \vct{y}) = \exp( t \|\vct{\hat{s}}\|^2 )  \sum_{\evec_i \in F} c_i \underbrace{\exp( t [\evec_i - \vct{\hat{s}}]^\intercal \vct{\hat{s}})}_{=1 \text{ for all } t} \exp( \evec_i^\intercal \vct{y}) . 
	\end{align}
	{\color{revisionBlue}Hence assuming $g^\star < 0$ means $\sum_{\evec_i \in F} c_i \exp( \evec_i^\intercal \vct{\hat{y}} ) < 0$ for some $\vct{\hat{y}}$ in $\text{Span}(\vct{s})^\perp$.}
	Using this $\vct{\hat{y}}$, one may verify that
	\begin{equation}
	\lim_{t \to \infty}f(t\vct{\hat{s}} + \vct{\hat{y}}) = -\infty,
	\end{equation}
	and so when $r$ is positive, $g^\star < 0$ implies $f^\star < 0$.
\end{proof}

\subsection{Proof of Proposition \ref{prop:strong_duality}}

In Section \ref{sec:SAGE} we asserted strong duality held between \eqref{eq:defSage0Primal} and \eqref{eq:defSage0Dual_conic}.
Here we prove a more general result with two lemmas. In what follows, ``$\mathrm{co}$'' is an operator that computes a set's conic hull.

\begin{lemma}
	Fix a closed convex cone $K$ in $\mathbb{R}^n$. If $\vct{a}$ in $K^\dagger$ is such that 
	\[
	X \doteq \{\vct{x} : \vct{a}^\intercal \vct{x} = 1,~ \vct{x} \text{ in } K  \}
	\]
	is nonempty, then $\cl\cone X = K$.\label{lem:strong_dual_cone}
\end{lemma}
\begin{proof}
	Certainly the conic hull of $X$ is contained within $ K$, and the same is true of its closure. 
	The task is to show that every $\vct{x}$ in $K$ also belongs to $\cl\cone X$;
	we do this by case analysis on $b \doteq \vct{a}^\intercal \vct{x}$. 
	
	By the assumptions $\vct{a} \in K^\dagger$ and $\vct{x} \in K$, we must have $b \geq 0$. 
	If $b$ is positive then the scaling $\tilde{\vct{x}} \doteq \vct{x} / b$ belongs to $K$ and satisfies $\vct{a}^\intercal \tilde{\vct{x}} = 1$.
	That is, $b > 0$ gives us $\tilde{\vct{x}}$ in $X$.
	Simply undo this scaling to recover $\vct{x}$ and conclude $\vct{x} \in \cone X$.
	Now suppose $b = 0$
	Here we consider the sequence of points $\vct{y}_n \doteq \vct{x}_0 + n \vct{x}$, where $\vct{x}_0$ is a fixed but otherwise arbitrary element of $X$.
	Each point $\vct{y}_n$ belongs to $K$, and has $\vct{a}^\intercal \vct{y}_n = 1$, hence the $\vct{y}_n$ are contained in $X$.
	It follows that the \textit{scaled points} $\vct{y}_n / n$ are contained in $\cl\cone X$, and the same must be true of their limit $\lim_{n \to \infty}\vct{y}_n / n = \vct{x}$.
	
	Since $\vct{x}$ in $K$ was arbitrary, we have $\cl\cone X = K$. 
\end{proof}

\begin{lemma}
	Let $C$ be a closed convex cone, and fix $\vct{a} \in C^\dagger \setminus \{\vct{0}\}$.
	Then the primal dual pair
	\[
	f_{\mathrm{p}} = \sup\{ \gamma : \vct{c} - \gamma \vct{a} \text{ in } C \}\quad \text{ and } \quad
	f_{\mathrm{d}} = \inf\{ \vct{c}^\intercal \vct{v} : \vct{a}^\intercal\vct{v} = 1,~ \vct{v} \text{ in } C^\dagger \}
	\]
	exhibits strong duality.\label{lem:general_sd}
\end{lemma}	
\begin{proof}
	By assumption that $\vct{a}$ is a nonzero vector in $C^\dagger$, the dual feasible set $\{  \vct{v} ~:~ \vct{a}^\intercal\vct{v} = 1,~ \vct{v} \text{ in } C^\dagger \}$ is nonempty. 
	Since the dual problem is feasible, a proof that $f_{\mathrm{d}} = f_{\mathrm{p}}$ can be divided into the cases $f_{\mathrm{d}} = -\infty$, and $f_{\mathrm{d}}$ in $\mathbb{R}$.
	The proof in former case is trivial; weak duality combined with $f_{\mathrm{p}} \geq -\infty$ gives $f_{\mathrm{d}} = f_{\mathrm{p}}$.
	In the latter case we prove $f_{\mathbf{p}} \geq f_{\mathbf{d}}$ by showing that $\vct{c}^\star \doteq \vct{c} - f_{\mathrm{d}}\vct{a}$ belongs to $C$.
	
	To prove $\vct{c}^\star \in C$ we will appeal to Lemma \ref{lem:strong_dual_cone} with $K \doteq C^\dagger$.
	Clearly the set $X = \{ \vct{v} : \vct{a}^\intercal\vct{v}=1, \vct{v} \text{ in } K  \}$ is precisely the [nonempty] feasible set for computing $f_{\mathrm{d}}$, and so from the definition of $f_{\mathrm{d}}$ we have ${\vct{c}^\star}^\intercal \vct{v} \geq 0 \text{ for all } \vct{v} \text{ in } X$.
	The inequality also applies to any $\vct{v}$ in $\cl\cone X$, which by Lemma \ref{lem:strong_dual_cone} is equal to $K^\dagger$.
	Therefore the definition of $f_{\mathrm{d}}$ ensures $\vct{c}^\star$ is in $K^\dagger$. Using $K^\dagger \equiv C$, we have the desired result. 
\end{proof}

Strong duality in computation of $f_{\mathsf{SAGE}}$ for $f = \Sig(\emat,\vct{c})$ readily follows from Lemma \ref{lem:general_sd}.
Letting $N = \{ i : c_i < 0\}$, simply take $C = \csage{\emat}$ or $C = \sum_{i \in N\cup \{1\}} \cage{\emat,i,N}$, and use $\vct{a} = \vct{e}_1$. What's more, with appropriate bookkeeping one can use Lemma \ref{lem:general_sd} to prove strong duality in computation of $\sagerelax{f}{p}$ for \textit{any} nonnegative integer $p$!

\subsection{Proof of Proposition \ref{prop:tfae}}

\begin{proof}
	The cases $(2) \Rightarrow (1)$ and $(3) \Rightarrow (1)$ are easy.
	
	$\neg (2) \Rightarrow  \neg (1)$. Because $\cnns{\emat}$ and $\csage{\emat}$ are full dimensional closed convex sets, the condition $\csage{\emat} \neq \cnns{\emat}$ implies that $\cnns{\emat} \setminus \csage{\emat}$ has nonempty interior. Assuming this condition, fix a vector $\tilde{\vct{c}}$ and a radius $r$ such that $B(\tilde{\vct{c}},r) \subset \cnns{\emat} \setminus \csage{\emat}$.\footnote{$B(\vct{x},d)$ is $\ell_2$ ball centered at $\vct{x}$ of radius $d$.} This allows us to strictly separate $\tilde{\vct{c}}$ from $\csage{\emat}$, which establishes $f^\star \geq f_{\mathsf{SAGE}} + r > f_{\mathsf{SAGE}}$.
	
	$(1) \Rightarrow (3)$. Now suppose that $f^\star = f_{\mathsf{SAGE}}$ for all relevant $f$. 
	In this case, the function $\vct{c} \mapsto \inf\{ \vct{c}^\intercal \vct{x} : \vct{x} \in \Omega\}$ is the same for $\Omega = \cl \conv \exp \mathcal{R}(\emat^\intercal)$ or $\Omega = \{ \vct{v} : v_1 = 1 \text{ and } \vct{v} \text{ in } \csage{\emat}^\star \}$. 
	This function completely determines the set of all half spaces containing $\Omega$. 
	Since $\Omega$ is closed and convex, it is precisely equal to the intersection of all half spaces containing it; the result follows. 
\end{proof}

\end{document}